\documentclass{article}

\usepackage[english]{babel}

\usepackage[letterpaper,top=2cm,bottom=2cm,left=3cm,right=3cm,marginparwidth=1.75cm]{geometry}

\usepackage{amsmath}
\usepackage{amsthm}
\usepackage{amssymb}
\usepackage{amsfonts}
\usepackage{graphicx}
\usepackage{float}
\usepackage{ulem}
\usepackage{todonotes}
\usepackage{subfigure}

\usepackage[colorlinks=true, allcolors=blue]{hyperref}

\newtheorem{definition}{\textbf{Definition}}[section]

\newtheorem{Thm}[definition]{\textbf{Theorem}}

\title{A comparative study of three mathematical approaches applied to the reversal of AMR}
\author{Sebastian Builes$^{1*}$,  Jhoana P Romero-Leiton$^2$ and   Leon A. Valencia$^{1}$ \\ \\
$^{1}$ Insitute of Mathematics, University of Antioquia, Medellín, Colombia \\
$^2$ Department of Mathematical Sciences, University of Puerto Rico at Mayagüez,\\ Puerto Rico, USA \\ \\
$^*$ Correspondence: lalexander.valencia@udea.edu.co  }
\date{}

\begin{document}
\maketitle

\begin{abstract}
In this work, we study the qualitative properties of a simple mathematical model inspired by antimicrobial resistance (AMR), focusing on the reversal of resistance. In particular, we analyze the model from three perspectives: ordinary differential equations (ODEs), stochastic differential equations (SDEs) driven by Brownian motion, and fractional differential equations (FDEs) with Caputo temporal derivatives. Finally, we perform numerical experiments using data from {\it Escherichia coli} exposed to colistin to assess the validity of the qualitative properties of the model. 

\end{abstract}

\section{Introduction}
 Mathematical modelling of biological systems is an important tool for describing and predicting their behaviour. Although these models are not ideal, their study is necessary because they can approximate the reality  under certain conditions. Generally, these models can be viewed from the deterministic, stochastic, and fractional perspectives. Accurately forecasting disease evolution over time is essential for effective modelling. Mathematical models have mostly been deterministic and highly idealized, assuming predictable progression of diseases based on fixed parameters. However, it soon became clear that the dynamics of infectious agents could be unpredictable and could be influenced by inherent population variability or memory effects. This realization has led to the development of stochastic models that incorporate randomness to better capture uncertainties in disease spread. As mathematical theory has advanced, fractional models have emerged, extending the scope of traditional models by incorporating fractional derivatives to account for complex real-world behaviors that are not adequately described by standard deterministic or stochastic approaches. This study aims to mathematically study the problem of antimicrobial resistance (AMR). Specifically, we focus on analysing the qualitative properties of a resistance reversal model, such as the existence, uniqueness, stability, invariance, and ergodicity of the solutions. 
\par 
 The mathematical study of epidemiological models began in 1927 with the introduction of the SIR model, proposed by Kermack and McKendrick \cite{Kermack}, which describes a closed population of susceptible, infected, and recovered individuals. In 1978, Capasso and Serio \cite{Capasso} extended this model by modifying the incidence rate to account for psychological effects on the infected population. The SIR model was later generalized to include births, deaths, and loss of immunity, resulting in the SIRS model, which allows recovered individuals to lose immunity and become susceptible. If the recovered individuals are assumed to have no immunity, the model is simplified to an SIS model. 
\par 
 With advancements in the probability theory, epidemiological models have increasingly incorporated probabilistic elements. Lu \cite{Lu} and Tornatore and Buccellato \cite{Tor} investigate the stability of equilibrium points in stochastic SIR models. Gray \cite{Gray},  Rifhat \cite{rifhat2020dynamical} and Xu \cite{Xu} examined the ergodic properties of the stochastic SIS models. These studies introduce stochastic perturbations to the disease transmission coefficient $\beta$ in the form $\widehat\beta dt = \beta dt + \sigma dB(t)$, where $\sigma$ represents the stochastic perturbation rate and $B(t)$ denotes Brownian motion. In addition, \cite{Mau} explored a fractional SIR model, extending the temporal derivative to a Caputo fractional derivative of order $\alpha$ with $\alpha \in (0,1)$.  
\par 
 Antimicrobial resistance (AMR) can be viewed as an epidemic model that illustrates the spread and impact of resistance patterns within a population of susceptible bacteria. AMR remains a major challenge in modern medicine \cite{andersson2010antibiotic}. It persists even in the absence of antibiotic pressure due to mechanisms such as horizontal gene transfer, particularly conjugation, which allows resistant bacteria to survive despite fitness costs \cite{lopatkin2017persistence}. This indicates that simply reducing antibiotic use may not be sufficient to reverse the resistance \cite{lopatkin2017persistence}.  However, AMR can be reversed because of the fitness costs associated with resistance, because sensitive bacteria may outcompete resistant bacteria in antibiotic-free environments \cite{andersson2010antibiotic}. Reversible mutations and the loss of resistance plasmids can also restore bacteria to a sensitive state, allowing sensitive populations to eventually dominate \cite{sundqvist2010little, bennett2008plasmid, andersson2011persistence}.  Various models have described the AMR dynamics, most of which are deterministic. However, to the best of our knowledge, no existing mathematical models simultaneously address resistance reversal using deterministic, stochastic, and fractional approaches.  
\par 
 This study is divided into two main parts. First, we analyze a mathematical model of AMR reversal deterministically, stochastically, and fractionally. The qualitative properties of the model are then analyzed. More specifically, we ensured the existence and uniqueness of solutions to these models. In addition, we ensured the invariance of the solutions within a biologically relevant set. We also demonstrated the stability of trivial solutions in these models. Finally, we provide the conditions for disease extinction, disease persistence, and stationary measures in some models. In the second part, we use the data from {\it E. coli} exposed to colistin to numerically validate the theoretical results.  

\section{Mathematical model formulation}
Let $ S(t) $ denote the number of sensitive bacteria at time $ t $,  and $ R(t) $ represent the number of resistant bacteria at time $t $. To formulate our mathematical model, we make the following basic assumptions: 
\begin{itemize}
    \item[-] The total bacterial population is constant at each time, such that $S(t) + R(t)=N$, where $ N $ is the total bacteria population.
    \item[-]  Let $ \mu $ represents the turnover rate of bacteria, and $ \beta $ denotes the rate at which sensitive bacteria acquire plasmids through conjugation. Resistant bacteria lose their resistance genes and revert to a sensitive state at a rate $ \gamma $.
    \item[-] Sensitive bacteria reduce the rate of acquiring resistance by selective pressure as the population of resistant bacteria increases, which is modeled by a functional response $g(R)=\dfrac{R}{1+\epsilon R} $.
    
\end{itemize}
Consequently, our mathematical model is described by the following deterministic equations:
\begin{equation}\label{model-gen}
    \left\{ \begin{array}{rcl}
              \dfrac{dS}{dt} &= &\mu(N-S)+\gamma R-\beta \dfrac{R}{1+{\epsilon} R}S \\ \\
              \dfrac{dR}{dt}&= &\beta \dfrac{R}{1+\epsilon R}S-(\gamma+\mu)R.
            \end{array} \right.
\end{equation}
In this work the term $\mathbb{R}_{0}^{+}$ represents non-negative real numbers. The equations in \eqref{model-det} are coupled by the assumption $N=S(t)+R(t)$ for all $t \in \mathbb{R}_{0}^{+}$. Thus, we can study one equation in terms of resistant bacteria $R(t)$. More precisely, we will study the deterministc equation:
\begin{equation}\label{model-det}
 \left\{ \begin{array}{rcl}
    \dfrac{dR}{dt} &=& \beta \dfrac{R}{1+\epsilon R}(N - R) - (\gamma + \mu)R.
    \end{array} \right.
\end{equation}
In the stochastic model, we perturb the rate $\beta$ as 
\begin{equation*}\label{sto-pert}
\widehat\beta dt = \beta dt + \sigma dB(t),
\end{equation*}
where $\sigma$ is the rate of stochastic perturbation and $B$ is a Brownian motion in a probability space $(\Omega,\mathcal{F},\mathbb{P})$. Therefore, the stochastic differential equation takes the form:
\begin{equation}\label{model-stoc}
\left\{ \begin{array}{rcl}
    dR &=& \left(\beta \dfrac{R}{1+\epsilon R}(N - R) - (\gamma + \mu)R\right) dt + \sigma \dfrac{R}{1+\epsilon R}(N - R) dB(t).
    \end{array} \right.
\end{equation}
This type of perturbation appears in \cite{Gray, Tor}.
\par
Finally, we introduce the fractional model, which we represent with the fractional differential equation:
\begin{equation}\label{model-fra}
\left\{ \begin{array}{rcl}
    \dfrac{d^{\alpha}R}{dt^{\alpha}} &=& \beta \dfrac{R}{1+\epsilon R}(N - R) - (\gamma + \mu)R,
    \end{array} \right.
\end{equation}
where $\dfrac{d^{\alpha}}{dt^{\alpha}}$ is the Caputo derivative of order $\alpha$, with $\alpha \in (0,1)$. This type of fractional model equation can be found in \cite{Mau} and \cite{Guo}.\\
In this work, we define threshold parameters that characterize the qualitative properties of (\ref{model-det}), (\ref{model-stoc}), and (\ref{model-fra}). The ideas for proving the qualitative properties of (\ref{model-stoc}) were adapted from \cite{builes2024stochastic} and are presented here in terms of the stochastic threshold parameter.
\section{Theoretical results}
\subsection{Defining thresholds}
In order to facilitate the mathematical calculations of the three models \eqref{model-det}-\eqref{model-fra} and make the results easier to understand, we defined certain thresholds, similar to the basic reproduction numbers in epidemiological models, and introduced Lyapunov operators. These thresholds and operators guide the stability and persistence of plasmid-carrying bacteria within a population in different modelling contexts.

\begin{table}[H]
\centering
\begin{tabular}{llll}
  \hline
  Approach & Threshold & Lyapunov operator& \\ \hline 
 Deterministic  &$\mathcal{K}_{0}^{d}:=\dfrac{\beta N}{(\gamma+\mu )}$ &$L_d:=\displaystyle\left(\beta\dfrac{R}{1+ \epsilon R}(N-R)-(\gamma+\mu)R\right)\frac{d}{dR}$ \\
 Stochastic &$ \mathcal{K}_{0}^{s}:=\dfrac{\beta N}{(\gamma+\mu)}-\dfrac{\sigma^2 N^2}{2(\gamma+\mu)} $&$L_s:=L_{d}+\dfrac{1}{2}\left(\sigma^{2}\dfrac{R^2}{(1+\epsilon R)^{2}}(N-R)^2\right)\dfrac{d^2}{dR^2} $\\
 Fractionary &$\mathcal{K}_{0}^{f}:=\dfrac{\beta N}{(\gamma+\mu)}$ & $L_f:=\dfrac{d^{\alpha}}{dt^{\alpha}}$ \\
  \hline 
\end{tabular}
\caption{Thresholds and Lyapunov operators for the three mathematical approaches considered in our mathematical models \eqref{model-det}-\eqref{model-fra}.}
\label{table2}
\end{table}
\subsection{The deterministic approach}
This section explores the qualitative properties of the deterministic model (\ref{model-det}). Let us note that if $\beta \dfrac{R}{1+\epsilon R}(N - R) - (\gamma + \mu)R=0$, the equilibrium points of (\ref{model-det}) are 
\begin{equation}\label{eq-det}
R=0 \quad \text{and} \quad R = \dfrac{\beta N - \gamma - \mu}{\beta + \epsilon(\gamma + \mu)}:=\xi_{d}. 
\end{equation}
We start showing that $(0,N)$ is an invariant set. More precisely, if $R_0 \in (0,N)$, then there exists a unique global solution  $R(t;R_0)$ of (\ref{model-det}) such that $R(t;R_0) \in (0,N)$ for all $t \in \mathbb{R}_{0}^{+}$. For this end, we use Lyapunov operator techniques, which are typically applied in the case of SDEs, but these techniques can also be used in the deterministic case, where the stochastic noise is zero (see \cite{khasminskii2011stochastic}).
 \begin{Thm}[{\bf Invariance}]
    For any $R_0 \in (0,N)$, there exists a unique global solution of (\ref{model-det}) invariant in $(0,N).$
 \end{Thm}
\begin{proof}
Let $R_0 \in (0,N)$. Consider $D := (0,N)$ and let $b(R) := \beta \dfrac{R}{1+\epsilon R}(N-R) - (\gamma + \mu)R$, for $R \in \mathbb{R}^{+}$. It is evident that $b$ is locally Lipschitz. Now, consider 
$V : D \to \mathbb{R}$ defined by: $V(R) := \dfrac{1}{R} + \dfrac{1}{N-R}$
and $D_n := \Big(\dfrac{1}{n}, N - \dfrac{1}{n}\Big)$, for $n \in \mathbb{N}$. It is clear that $V \in C^1( D; \mathbb{R}_{0}^{+})$ and that $\{D_n\}_{n \in \mathbb{N}} \subset D$, $D_n$ are domains, $\overline{D_n} \subset D$, $D_n \uparrow D$ such that $V_n := \displaystyle \inf_{R \notin D_n} {V(R)} \to \infty, \quad \text{as } n \to \infty.$ On the other hand, given $R \in D$, we have that:
\begin{equation*}
  \begin{aligned}
    L_d V(R) &= b(R) V'(R)\\
    &= \left(\beta \dfrac{R}{1+\epsilon R}(N-R) - (\gamma + \mu)R\right) \left(\dfrac{-1}{R^2} + \dfrac{1}{(N-R)^2}\right)\\
    &= \dfrac{-\beta (N-R)}{R(1+\epsilon R)} + \dfrac{(\gamma + \mu)}{R} + \dfrac{\beta R}{(N-R)(1+\epsilon R)} - \dfrac{(\gamma + \mu)R}{(N-R)^2}\\
    & \le \dfrac{(\gamma + \mu)}{R} + \dfrac{\beta N }{(N-R)}.\\
  \end{aligned}   
\end{equation*}

\noindent Let us take $c := \max \{\gamma + \mu, \beta N \}$. Thus,
$$L_d V(R) \leq c \left(\dfrac{1}{R} + \dfrac{1}{N-R}\right) = c V(R).$$

\noindent Then, there exists a unique solution $R(\cdot\ ; R_0)$ of (\ref{model-det}) in $\mathbb{R}_{0}^{+}$ such that $R(t;R_0) \in (0, N), \text{for all}\ t \in \mathbb{R}_{0}^{+}$. \qed 
\end{proof}
\vspace{0.3cm}
Now, we prove the asymptotic stability of the equilibrium points of the deterministic model \eqref{model-det} defined in \eqref{eq-det} (see definition \ref{deterministic_type}).
Intuitively, the parameter $\mathcal{K}_{0}^{d}$ acts as an indicator of the system's dynamics, as we will see in the following theorem.

\begin{Thm}[{\bf Asymptotic stability}]
If $\mathcal{K}_{0}^{d} < 1$, then $R = 0$ is an asymptotically stable equilibrium point of (\ref{model-det}). Moreover, if $\mathcal{K}_{0}^{d} > 1$, then $R = \xi_{d}$ is an asymptotically stable equilibrium point of (\ref{model-det}). 
\end{Thm}
\begin{proof}
First, let's see that $R = 0$ is an asymptotically stable equilibrium point of (\ref{model-det}). Consider $\delta > 0$ small enough and let $D = (-\delta, \delta)$. Note that $b(R) := \beta \dfrac{R}{1+\epsilon R}(N-R) - (\gamma + \mu)R$, for $R \in D$. It is clear that $b$ is differentiable and $b'(R) = \beta \dfrac{1}{(1+\epsilon R)^2}(N-R) - \beta \dfrac{R}{1+ \epsilon R} - (\gamma + \mu)$, for $R \in D$. Thus, $b'(0) = \beta N - (\gamma + \mu) < 0$. By the Lyapunov linearization theorem for the deterministic case, we have that $R = 0$ is an asymptotically stable equilibrium point of (\ref{model-det}).\\
Now, let's see that $R = \xi_{d}$ is an asymptotically stable equilibrium point of (\ref{model-det}). Let $D := (0, N)$ and consider $V: D \to \mathbb{R}$ defined by:
\[
V(R) := \dfrac{\xi_{d}}{N-\xi_{d}}\left(\dfrac{N}{R} - 1\right) - 1 - \ln\left(\dfrac{\xi_{d}}{N-\xi_{d}} \left(\dfrac{N}{R} - 1\right)\right).
\]
Clearly, $V \in C^1(D; \mathbb{R})$ and $V$ is positive definite at $R = \xi_{d}$. 
Also, given $R \in D$, we have:
\begin{align*}
    L_d V(R) & = b(R) V'(R) \\
    & = \left(\beta \dfrac{R}{1+ \epsilon R}(N-R) - (\gamma + \mu)R\right)\left(-\dfrac{N^2}{R^2}\dfrac{(\xi_{d} - R)}{(N - \xi_{d})(N - R)}\right)\\
    & = \left(\beta \dfrac{N-R}{1+\epsilon R} - (\gamma + \mu)\right)\left(-\dfrac{N^2}{R}\dfrac{(\xi_{d} - R)}{(N - \xi_{d})}\right).
\end{align*}
Thus, $L_d V(R) < 0$, for $R \neq \xi_{d}$.
Then, by the Lyapunov stability theorem for equilibrium points in the deterministic case, we have that $R = \xi_{d}$ is an asymptotically stable equilibrium point of (\ref{model-det}). \qed
\end{proof}
\vspace{0.3cm}
In the following theorem, we show that if the deterministic threshold condition $\mathcal{K}_{0}^{d} < 1$ is fulfilled, the population of resistant bacteria will ultimately die out. Of course, this result is stronger than the previous one when $\mathcal{K}_{0}^{d} < 1$
\begin{Thm}[{\bf Extinction}]
   If $\mathcal{K}_{0}^{d} < 1$, then for any $R_0 \in (0, N)$, we have
   $$\displaystyle\lim_{t \to \infty} R(t; R_0) = 0.$$
\end{Thm}
\begin{proof}
Let's show that
$$\displaystyle\limsup_{t \to \infty} \frac{\ln(R(t; R_{0}))}{t} < 0.$$
We already know that there exists a unique solution $R(\cdot; R_0)$ in $\mathbb{R}_{0}^{+}$ of (\ref{model-det}) such that $R(t; R_0) \in (0, N)$ for all $t \in \mathbb{R}_{0}^{+}$. Let us denote $R(\cdot; R_0)$ by $R$ and $D := (0, N)$.\\

Now, consider the function $V : D \to \mathbb{R}$ defined by $V(R) := \ln(R)$. It is clear that $V \in C^{1}(D; \mathbb{R})$.\\
Let $t > 0$. By the chain rule and integrating from 0 to $t$, we have:
$$\displaystyle\int_{0}^t \frac{dV(R)}{ds}(s) ds = \displaystyle\int_{0}^t L_{d} V(R(s)) ds.$$
In other words,
$$V(R(t)) - V(R(0)) = \displaystyle\int_{0}^t L_{d} V(R(s)) ds.$$
That is,
\begin{equation*}
    \begin{aligned}
    \ln(R(t)) &= \ln(R_0) + \displaystyle\int_{0}^t L_{d} V(R(s)) ds \\
    &= \ln(R_0) + \int_{0}^t \left(\beta \dfrac{R(s)}{1+\epsilon R(s)}(N - R(s)) - (\gamma + \mu)R(s)\right) \frac{1}{R(s)} ds \\
    &= \ln(R_0) + \int_{0}^t \left(\beta \frac{N - R(s)}{1 +\epsilon R(s)} - (\gamma + \mu)\right) ds \\
    &\leq \ln(R_0) + (\beta N - (\gamma + \mu))t.
    \end{aligned}
\end{equation*}

Then, dividing by $t$, we get
$$\frac{\ln(R(t))}{t} \leq \frac{\ln(R_0)}{t} + \beta N - (\gamma + \mu).$$
Therefore,
$$\limsup_{t \to \infty} \frac{\ln(R(t))}{t} \leq \beta N - (\gamma + \mu) < 0.$$
Thus,
$$\displaystyle\lim_{t \to \infty} R(t) = 0.$$ \qed
\end{proof}
In the following theorem, we state that when the deterministic threshold $\mathcal{K}_{0}^{d} > 1$, the population of resistant bacteria will persist over time. In other words, the system described by \eqref{model-det} converges to the non-trivial equilibrium point $\xi_{d}$ defined in \eqref{eq-det} as $t \to \infty$. This implies that $\xi_{d}$ is no longer an equilibrium point but rather an attractor for the system. This result is significant because, despite $\xi_{d}$ not being an equilibrium point, it continues to attract the system's trajectories as $t \to \infty$.
\begin{Thm}[{\bf Persistence}]
If $\mathcal{K}_{0}^{d} > 1$, then there exists $\xi_{d} \in (0, N)$ such that for any $R_0 \in (0, N)$,
\begin{equation}\label{lims}
   \displaystyle\lim_{t \to \infty} R(t; R_0) = \xi_{d},
\end{equation}
where $\xi_{d}$ is the equilibrium point defined in \eqref{eq-det}.
\end{Thm}

\begin{proof}
Let $D := (0, N)$ and consider $V : D \to \mathbb{R}$ defined by:
\[
V(R) := \dfrac{\xi_{d}}{N - \xi_{d}}\left(\dfrac{N}{R} - 1\right) - 1 - \ln \left(\dfrac{\xi_{d}}{N - \xi_{d}} \left(\dfrac{N}{R} - 1\right)\right).
\]
Clearly, $V \in C^{1}(D; \mathbb{R})$, $V$ is positive definite at $R = \xi_{d}$, and $L_{d} V(R) < 0$ for $R \neq \xi_{d}$.
Additionally, we have:
$V(R) \to \infty$ as $R \to 0^{+}$ and as $R \to N^{-}$.
By the Barbashin-Krasovskii theorem (Theorem 4.2 in \cite{hassan2002nonlinear}), $R = \xi_{d}$ is a globally asymptotically stable equilibrium point of (\ref{model-det}), meaning that $\displaystyle\lim_{t \to \infty} R(t; R_0) = \xi_{d}$ for any $R_{0} \in (0, N)$. 
\qed
\end{proof}

\subsection{The stochastic approach}

In this section, we present a series of results involving the numerical value \( \mathcal{K}_0^s \). The proofs are direct modifications of the proofs in \cite{builes2024stochastic}. In \cite{builes2024stochastic}, a family of stochastic differential equations is considered, depending on a general function \( g(R) \) under certain conditions. In this article, we work with the function \( g(R) = \frac{R}{1 + \epsilon R} \) as a particular case of \cite{builes2024stochastic}. We have chosen to provide intuitive commentary on the results and to reference the work in \cite{builes2024stochastic} for the full proofs.

Now, in this section explores the qualitative properties of the stochastic model (\ref{model-stoc}). Let us note that $\beta \dfrac{R}{1+\epsilon R}(N - R) - (\gamma + \mu)R=0$ and $\sigma \dfrac{R}{1+\epsilon R}(N - R)=0$, the only equilibrium point of (\ref{model-stoc}) is $R=0.$ 


The following result is fundamental in the stochastic modelling of population systems, as it guarantees that, regardless of the initial condition \( R_0 \in (0, N) \), the stochastic system has a unique global solution that remains within the interval \( (0, N) \) over time. This implies that the model is robust to random fluctuations and ensures that the stochastic dynamics of the resistant bacterial population neither becomes extinct in finite time nor exceeds the maximum capacity \( N \). In other words, the system is able to realistically model a sustained population range in the presence of randomness.

\begin{Thm}[{\bf Invariance}]
 For any \( R_0 \in (0, N) \), there exists a unique global solution of the stochastic system (\ref{model-stoc}) that remains invariant in \( (0, N) \).
\end{Thm}
\begin{proof}
See Theorem 2.1 in \cite{builes2024stochastic}.
\qed
\end{proof}
\vspace{0.3cm}
 The following result states that if \( \mathcal{K}_0^s < 1 \), then, starting near extinction ($R=0$), the resistant bacterial population converges to extinction with high probability (see definition \ref{type_stochastic}).

\begin{Thm}[{\bf Asymptotic stability}]
  If $\mathcal{K}_{0}^{s}<1$, then \( R=0 \) is an asymptotically stable equilibrium point in probability of (\ref{model-stoc}).
\end{Thm}
\begin{proof}
  See Theorem 2.2 in \cite{builes2024stochastic}.
  \qed
\end{proof}
\vspace{0.3cm}
In the following theorem, we present sufficient conditions for the population of resistant bacteria to converge to extinction as \( t \to \infty \). One condition that seems natural after the previous results is that \( \mathcal{K}_{0}^{s} < 1 \); this ensures that if we start close to the equilibrium point \( R = 0 \), the system will converge to \( 0 \) with high probability. Another condition is needed to guarantee convergence to extinction starting from any point. This condition arises from the Lyapunov function used, so it may be subject to refinement. The condition is that \(\sigma^2 N \leq \beta \); this suggests a certain restriction on \( \beta \), in the sense that this parameter cannot be taken arbitrarily small.


\begin{Thm}[{\bf Extinction}]
  If \( \mathcal{K}_{0}^{s} < 1 \) and \(\sigma^2 N \leq \beta \), then for any \( R_0 \in (0, N) \), we have:
  \[
  \mathbb{P}\left(\lim_{t \to \infty} R(\cdot, t; R_{0}) = 0 \right) = 1.
  \]
\end{Thm}
 \begin{proof}
We already know that there exists a unique solution $R(\cdot,\cdot;R_0)$ in $\mathbb{R}_{0}^{+}$ of (\ref{model-stoc}) such that $R(\cdot,t;R_0) \in (0,N),$ for all $t \in \mathbb{R}_{0}^{+},$ $\mathbb{P}$-a.s. Let us denote $R(\cdot,\cdot;R_0)$ by $R$ and define $D := (0, N).$\\

Now, consider the functions $V: D \to \mathbb{R}$ defined by $V(R) := \ln(R)$ and $\varphi: D \to \mathbb{R}$ defined by $\varphi(R) := \dfrac{N - R}{1 + R}.$ It is clear that $V \in C^{2}(D; \mathbb{R})$ and $\varphi(R) \le N,$ for all $R \in (0, N).$\\

\noindent Let $t > 0.$ By Itô's formula, we have:
$$V(R(t)) - V(R(0)) = \int_{0}^{t} L_{s}V(R(s)) \, ds + \int_{0}^{t} V'(R(s)) \sigma(R(s)) \, dB(s), \quad \mathbb{P}\text{-a.s.}$$

\noindent In other words,
\begin{equation*}
    \begin{aligned}
        \ln(R(t)) &= \ln(R_0) + \int_{0}^{t} L_{s}V(R(s)) \, ds + \int_{0}^{t} V'(R(s)) \sigma(R(s)) \, dB(s)\\
        &= \ln(R_0) + \int_{0}^{t} \left( \beta \frac{N - R(s)}{1 + R(s)} - (\gamma + \mu) - \frac{1}{2} \frac{\sigma^{2} (N - R(s))^2}{(1 + R(s))^{2}} \right) ds\\
        &\quad + \int_{0}^{t} \sigma \frac{N - R(s)}{1 + R(s)} \, dB(s)\\
        &= \ln(R_0) + \int_{0}^{t} \beta \varphi(R(s)) - (\gamma + \mu) - \frac{1}{2} \sigma^2 \varphi^2(R(s)) \, ds + \int_{0}^{t} \sigma \varphi(R(s)) \, dB(s).
    \end{aligned}
\end{equation*}

Now, since $N \le \dfrac{\beta}{\sigma^2}$ and the function $f(u) := \beta u - (\gamma + \mu) - \dfrac{1}{2} \sigma^{2} u^2$ is increasing for $u \le \dfrac{\beta}{\sigma^2},$ then $f(\varphi(R(s))) \le f(N),$ for all $s \in [0, t].$ In other words,\\

\[
\beta \varphi(R(s)) - (\gamma + \mu) - \frac{1}{2} \sigma^2 \varphi^2(R(s)) \le \beta N - (\gamma + \mu) - \frac{1}{2} \sigma^2 N^2,
\]
for all $s \in [0, t].$
The rest of the proof is similar to the proof of Theorem 2.3 in \cite{builes2024stochastic}.

\qed
\end{proof}
\vspace{0.3cm}
Before stating the theorem on the persistence of the bacterial population, we would like to provide a brief intuition of the result. In this context, the quantity \( R(\cdot, t; R_0) \) represents the size of the resistant bacterial population over time \( t \), starting from an initial condition \( R_0 \). A key objective is to understand whether this population tends to persist over time or, on the contrary, tends to vanish.

The theorem below indicates that, under certain conditions on the system parameters (\( \mathcal{K}_{0}^{s} > 1 \) ), there exists a persistence level \( \xi_{s} \) such that the resistant bacterial population will repeatedly reach this level over time. In other words, despite the stochastic variability in the model, we can expect the population not only to remain above extinction thresholds but also to oscillate around \( \xi_s > 0 \). In summary, this implies a tendency towards persistence rather than extinction, providing a measure of stability for the resistant bacterial population under the effect of random factors.

The proof of the following theorem has subtle differences compared to the proof of Theorem 2.4 in \cite{builes2024stochastic}. We will highlight these differences in the demonstration, and when the tools are fully analogous, we will refer to \cite{builes2024stochastic}.

\begin{Thm}\label{persistence}
If \( \mathcal{K}_{0}^{s} > 1 \) 
    , then there exists \( \xi_{s} \in (0, N) \) such that for any \( R_0 \in (0, N) \), the following holds:
\begin{equation}\label{lims}
   \limsup_{t \to \infty} R(\cdot, t; R_0) \geq \xi_{s}, \quad \mathbb{P}-\text{a.s.}
\end{equation}
and
\begin{equation}\label{limi}
   \liminf_{t \to \infty} R(\cdot, t; R_0) \leq \xi_{s}, \quad \mathbb{P}-\text{a.s.}
\end{equation}
where \( \xi_{s} \) is given by
\[
\xi_{s} = \frac{\sigma^{2}N - \beta - \sqrt{\beta^2 - 2\sigma^2(\gamma + \mu)}}{\sigma^2 +\epsilon\beta +\epsilon \sqrt{\beta^2 - 2\sigma^2(\gamma + \mu)}}.
\]
In other words, the solution process \( R(\cdot,\cdot\ ; R_0) \) visits the level \( \xi_{s} \) infinitely often for any \( R_0 \in (0, N) \).
\end{Thm}

\begin{proof}
    Let $D:=(0,N)$, and consider $V: D\to \mathbb{R}$ defined by $V(R):=\ln(R).$ It is clear that $V\in C^{2}(D;\mathbb{R}).$ Note that the stochastic Lyapunov operator applied to $V$ is given by:
$$L_{s}V(R):=\dfrac{\beta (N-R)}{1+\epsilon R} -(\gamma+\mu)-\dfrac{1}{2}\dfrac{\sigma^{2}(N-R)^2}{(1+\epsilon R)^2},\ \ \text{for all}\ \ R\in D.$$
Rewriting, we have: 
$$L_{s}V(R)=\beta\varphi(R)-(\gamma+\mu)-\dfrac{1}{2}{\sigma^{2}}\varphi^{2}(R),\ \ \text{for all}\ \ R\in D.$$
Where $\varphi(R):=\dfrac{N-R}{1+\epsilon R}.$ 
Consider the quadratic function: $$f(u):=\beta u-(\gamma+\mu)-\dfrac{1}{2}{\sigma^{2}}u^2, \ \ \text{for all}\ \ u\in\mathbb{R}. $$
 Note that $f(0)=-(\gamma+\mu)$, and since $\mathcal{K}_{0}^{s}>1$, then $f(N)=\beta N -(\gamma+\mu)-\dfrac{1}{2}{\sigma^{2}}N^2>0.$ Therefore, there exists $\eta\in (0,N)$ such that $f(\eta)=0.$ Now, there exists a unique $\xi_{s}\in(0,N)$ such that $\varphi(\xi_{s})=\eta.$ Thus, $0=f(\eta)=f(\varphi(\xi_{s}))=L_{s}V(\xi_{s}).$\\
 
\textbf{Case 1.} If $\dfrac{\beta }{\sigma^{2}}<N.$ Since $\varphi'(R)<0,$ for all $R\in(0,N),$ there exists $m\in (0,N)$ such that:
\begin{enumerate}
 \item[-] $L_{s}V$  is increasing in $(0,m)$ and $L_{s}V(R)>0,$ for all $R\in(0,m).$
 \item[-] $L_{s}V$  is decreasing in $(m,\xi_{s})$ and $L_{s}V(R)>0,$ for all $R\in(m,\xi_{s}).$
 \item[-]  $L_{s}V$ is decreasing in $(\xi_{s},N)$ and $L_{s}V(R)<0,$ for all $R\in(\xi_{s},N).$
\end{enumerate}

Where $L_{s}V(m)$ is the maximum value of $L_{s}V$ in $(0,N).$\\

\textbf{Case 2.} If $\dfrac{\beta }{\sigma^{2}}\ge  N.$ Since $\varphi'(R)<0,$ for all $R\in(0,N),$ $L_{s}V(0^{+}):=\displaystyle\lim_{R\to 0^{+}}L_{s}V(R)=\beta N -(\gamma+\mu)-\dfrac{1}{2}{\sigma^{2}}N^2>0$ and $L_{s}V(N^{-}):=\displaystyle\lim_{R\to N^{-}} L_{s}V(R)=-(\gamma +\mu)<0,$ then: 
\begin{enumerate}
 \item[-] $L_{s}V$  is decreasing in $(0,\xi_{s})$ and $L_{s}V(R)>0,$ for all $R\in(0,\xi_{s}).$
 \item[-]  $L_{s}V$ is decreasing in $(\xi_{s},N)$ and $L_{s}V(R)<0,$ for all $R\in(\xi_{s},N).$
 
\end{enumerate}
The rest of the proof is similar to proof of Theorem 2.4 of \cite{builes2024stochastic}. 
\qed
\end{proof}
\vspace{0.3cm}
The proof of the above theorem can be found in \cite{builes2024stochastic}.
The following theorem establishes the stationary distribution of the solution $R$ of the stochastic system \eqref{model-stoc}. This allows us to ensure that the stochastic model exhibits stable and predictable behaviour over the long term under certain conditions. Therefore, as the system evolves, the distribution of resistant bacteria will stabilize to a unique stationary distribution. 


\begin{Thm}[{\bf Stationary distribution}]
    If $\mathcal{K}_{0}^{s}>1$ 
    , then $\mu_{\infty}(\cdot)$ is the unique stationary distribution associated with the solution $R$ of (\ref{model-stoc}).
\end{Thm}
\begin{proof}

    See Theorem 2.5 in \cite{builes2024stochastic} \qed.
\end{proof}
\vspace{0.3cm}
\subsection{The fractionary approach}
Finally, in this section explores the qualitative properties of the fractionary model (\ref{model-fra}). Let us note that if $\beta \dfrac{R}{1+\epsilon R}(N - R) - (\gamma + \mu)R=0$, the equilibrium points of (\ref{model-fra}) are 
\begin{equation}\label{eq-fra}
R=0 \quad \text{and} \quad R = \dfrac{\beta N - \gamma - \mu}{\beta + \epsilon(\gamma + \mu)}:=\xi_{f}.
\end{equation}
We start showing that $(0,\mathcal{K}_{0}^{f})$ is an invariant set. In this case the invariance is more restrictive, since the fractional model has more limitations.
\begin{Thm}[{\bf Invariance}]
    For any $R_0 \in (0,\mathcal{K}_{0}^{f})$, there exists a unique global solution of (\ref{model-fra}) invariant in $(0,\mathcal{K}_{0}^{f}).$
 \end{Thm}
\begin{proof}

\noindent Let $R_0 \in \big(0, \mathcal{K}_{0}^{f}\big)$ and define $b(R):=\beta\dfrac{R}{1+ \epsilon R}(N-R)-(\gamma+\mu)R$, for $R \in \mathbb{R}^{+}.$\\
First, let us show that there exists a unique global positive solution of (\ref{model-fra}). Since $b$ is locally Lipschitz, there exists a unique local maximal solution $R(\cdot;R_{0},\alpha):[0, T_{f}) \to \mathbb{R}$, where $T_{f}$ is the fractional explosion time. Since $R_{0} \in \mathbb{R}^{+}$, there exists $n_0 \in \mathbb{N}$ such that $\dfrac{1}{n_{0}} < R_{0}.$\\
Now, we define the following sequence of stopping times.
 $$T_n := \inf\Big\{t \in [0,T_{f}): R(t;R_{0},\alpha) \notin \Big(\dfrac{1}{n},n\Big)\Big\},\ \text{for}\ n > n_0.$$
 It is clear that $\{T_{n}\}_{n \ge n_0}$ is an increasing sequence. Moreover,  $$T_{\infty} := \lim_{n\to\infty}T_n = \inf\Big\{t \in [0,T_{f}): R(t;R_{0},\alpha) \notin (0,\infty) \Big\}.$$
 Let us show that $T_{\infty} = \infty.$ Suppose that $T_{\infty} < \infty$, then there exists $T \in \mathbb{R}^{+}$ such that $T_{\infty} \le T.$ Then, there exists $n_1 \in \mathbb{N},$ with $n_1 > n_0$ such that $T_{n} \le T.$
On the other hand, consider the function $V: \mathbb{R}^{+} \to \mathbb{R}$ defined by: 
$$V(R):=R-1-\ln(R).$$

\noindent Given $R \in \mathbb{R}^{+},$ by property lemma 3.1 of \cite{vargasdeleon2015volterra}, we have:

\begin{equation*}
\begin{aligned}
  L_{f}V(R) &\le \Big(1 - \frac{1}{R}\Big)\frac{d^{\alpha}R}{dt^{\alpha}}\\
  &=\Big(1 - \frac{1}{R}\Big)\Big(\beta \frac{R}{1+\epsilon R}(N-R) - (\gamma+\mu)R\Big)\\
   &=\beta (N-R) - (\gamma+\mu)R - \frac{\beta (N-R)}{1+\epsilon R} + (\gamma+\mu)\\
   &\le \beta N + \beta + \gamma + \mu.
\end{aligned}
\end{equation*}

\noindent Let $c := \beta N + \beta + \gamma + \mu.$ Thus, for all $R \in \mathbb{R}^{+}$, $L_{f}V(R) \le c.$ Applying the fractional integral to both sides from 0 to $T \wedge T_n$, for $n \in \mathbb{N}$, with $n > n_0$, we have:
\begin{equation*}
 V(R(T \wedge T_n)) \le V(R_0) + \frac{cT^{\alpha}}{\Gamma(\alpha+1)}.
\end{equation*}
Now, for all $n \in \mathbb{N}, n > n_1,$ we have $T \wedge T_n = T_n$ and $V(R(T \wedge T_n)) = V(R(T_n)).$\\ Note that:
\begin{equation*}
V(R(T_n)) = \Big(\frac{1}{n} - 1 - \ln\Big(\frac{1}{n}\Big)\Big) \wedge \Big(n - 1 - \ln(n)\Big),
\end{equation*}
hence, 
\begin{equation*}
\Big(\frac{1}{n} - 1 - \ln\Big(\frac{1}{n}\Big)\Big) \wedge \Big(n - 1 - \ln(n)\Big) \le V(R_0) + \frac{cT^{\alpha}}{\Gamma(\alpha+1)}.
\end{equation*}

\noindent Then, as $n \to \infty$, we have:
\begin{equation*}
V(R_0) + \dfrac{cT^{\alpha}}{\Gamma(\alpha+1)} = \infty.
\end{equation*}

\noindent This is a contradiction, since the expression $V(R_0) + \dfrac{cT^{\alpha}}{\Gamma(\alpha+1)}$ is finite.\\
Now, let us show that $R(t;R_{0},\alpha) \in (0,N),$ for all $t \in \mathbb{R}_{0}^{+}.$ Let us denote this solution by \( R \).\\
Let $t \in \mathbb{R}^{+}.$ Note that:
 \begin{equation*}
\begin{aligned}
  \frac{d^{\alpha}R}{dt^{\alpha}}(t) &= 
  \beta \frac{R(t)}{1+\epsilon R(t)}(N-R(t)) - (\gamma+\mu)R(t)\\
   &\le \beta N - (\gamma+\mu)R(t)
\end{aligned}
\end{equation*}
Solving the inequality, we have:
\begin{equation*}
\begin{aligned}
  R(t) &\le \Big(R_{0} - \mathcal{K}_{0}^{f}\Big) E_{\alpha}(-(\gamma+\mu)t^{\alpha}) + \mathcal{K}_{0}^{f},\\
\end{aligned}
\end{equation*}
where $E_{\alpha}$ is the Mittag-Leffler function of parameter $\alpha$ defined in \cite{Kilbas2006}. Since $0 \le E_{\alpha}(-(\gamma+\mu)t^{\alpha}) \le 1$ and $R_{0} - \mathcal{K}_{0}^{f} < 0,$ then 
\begin{equation*}
\begin{aligned}
  R(t) < \mathcal{K}_{0}^{f}.\\
\end{aligned}
\end{equation*}
\end{proof}

 \begin{Thm}[{\bf Asymptotic stability}]
If $\mathcal{K}_{0}^{f} < 1$, then $R = 0$ is an asymptotically stable equilibrium point of (\ref{model-fra}). Moreover, if $\mathcal{K}_{0}^{f} > 1$, then $R = \xi_{f}$ is an asymptotically stable equilibrium point of (\ref{model-fra}). 
\end{Thm}
\begin{proof}
 First, let us show that $R=0$ is an asymptotically stable equilibrium point of (\ref{model-fra}). Consider $\delta > 0$ sufficiently small and let $D = (-\delta, \delta).$ Note that $b(R) := \beta \dfrac{R}{1+ \epsilon R}(N-R) - (\gamma + \mu)R,$ for $R \in D.$ It is clear that $b$ is differentiable and $b'(R) = \beta \dfrac{1}{(1+\epsilon R)^{2}} (N-R) - \beta \dfrac{R}{1+ \epsilon R} - (\gamma + \mu),$ for $R \in D.$ Thus, $b'(0) = \beta N - (\gamma + \mu) < 0.$ By the fractional Lyapunov linearization theorem (Theorem 2 in \cite{matignon1996stability}), we have that $R=0$ is an asymptotically stable equilibrium point of (\ref{model-fra}).\\

Now, let us show that $R=\xi_{f}$ is an asymptotically stable equilibrium point of (\ref{model-fra}). Consider $\delta > 0$ sufficiently small and let $D = (\xi_{f} - \delta, \xi_{f} + \delta).$ Note that $b(R) := \beta \dfrac{R}{1+ \epsilon R}(N-R) - (\gamma + \mu)R,$ for $R \in D.$ It is clear that $b$ is differentiable and $b'(R) = \beta \dfrac{1}{(1+\epsilon R)^{2}} (N-R) - \beta \dfrac{R}{1+ \epsilon R} - (\gamma + \mu),$ for $R \in D.$ Thus, $b'(\xi_{f}) = \beta \dfrac{1}{(1+\epsilon\xi_{f})^{2}} (N-\xi_{f}) - \beta \dfrac{\xi_{f}}{1+\epsilon\xi_{f}} - (\gamma + \mu) < 0.$ By the fractional Lyapunov linearization theorem  (Theorem 2 in \cite{matignon1996stability}), we have that $R=\xi_{f}$ is an asymptotically stable equilibrium point of (\ref{model-fra}).
  
\end{proof}

\begin{Thm}[{\bf Extinction}]
   If $\mathcal{K}_{0}^{f} < 1$, then for any $R_0 \in (0,\mathcal{K}_{0}^{f})$, we have
   $$\displaystyle\lim_{t \to \infty} R(t; R_0,\alpha) = 0.$$
\end{Thm}
\begin{proof}

\noindent Let $D := [0, \mathcal{K}_{0}^{f}]$ and consider $V: D \to \mathbb{R}$ defined by:
\[
 V(R) := \frac{R^{2}}{2},\ \text{for}\ R \in D.
\]
Clearly, $V \in C^{1}(D; \mathbb{R})$ and $V$ are positive definite at $R = 0$. 
Also, given $R \in D$, using the Lemma 2.1 of \cite{vargasdeleon2015volterra} we have:
\begin{align*}
    L_{f}V(R) & \le  
  \frac{d^{\alpha}R}{dt^{\alpha}}\\
  &= \beta \frac{R}{1+\epsilon R}(N-R) - (\gamma + \mu)R\\
    & = R \Big(\dfrac{\beta (N-R)}{1+\epsilon R} - (\gamma + \mu)\Big) \\
     & = R \frac{(\beta+\epsilon(\gamma + \mu) )}{1+\epsilon R} (\xi_{f} - R) \\
     & = -R \frac{(\beta+ \epsilon(\gamma + \mu) )}{1+\epsilon R}(|\xi_{f}| + R). \\
\end{align*}
\noindent Thus, $L_{f}V(R) \le 0$ for all $R \in D,$ and moreover, $$\mathcal{A} := \{R \in D : L_{f}V(R) = 0\} = \{0\}.$$
Then, by the fractional LaSalle theorem (Lemma 4.6 in \cite{huo2015effect}), we have that $R = 0$ is a globally asymptotically stable equilibrium point of (\ref{model-fra}).  Then,  $\displaystyle\lim_{t \to \infty} R(t; R_0,\alpha) = 0 $ for any $R_{0} \in (0, \mathcal{K}_{0}^{f})$.\\

\end{proof}
\begin{Thm}[{\bf Persistence}]
If $\mathcal{K}_{0}^{f} > 1$ and $\dfrac{\beta N-\gamma-\mu}{\beta N}-\dfrac{\beta}{\gamma+\mu}<\epsilon<1$, then there exists $\xi_{f} \in (0,\mathcal{K}_{0}^{f})$ such that for any $R_0 \in (0,\mathcal{K}_{0}^{f})$,
\begin{equation}\label{lims}
   \displaystyle\lim_{t \to \infty} R(t; R_0,\alpha) = \xi_{f},
\end{equation}
where $\xi_{f}$ is the equilibrium point defined in \eqref{eq-det}.
\end{Thm}
\begin{proof}
 \noindent Let $D := [0, \mathcal{K}_{0}^{f}]$ and consider $V: D \to \mathbb{R}$ defined by:
\[
 V(R) := R - \xi_{f} - \xi_{f} \ln \left( \frac{R}{\xi_{f}} \right), \ \text{for}\ R \in D, R \neq 0, \ \text{and}\ \ V(0) = 1.
\]
Clearly, $V \in C^{1}(D - \{0\}; \mathbb{R})$ and $V$ is positive definite at $R = \xi_{f}$. 
Also, given $R \in D - \{0\}$, using the Lemma 3.1 of \cite{vargasdeleon2015volterra} we have:
\begin{align*}
    L_{f}V(R) & \le  
    \Big(1 - \frac{\xi_{f}}{R}\Big) \frac{d^{\alpha} R}{dt^{\alpha}} \\
  &= \Big(1 - \frac{\xi_{f}}{R}\Big) \left( \beta \frac{R}{1+\epsilon R} (N - R) - (\gamma + \mu) R \right) \\
    & = (R - \xi_{f}) \Big( \dfrac{\beta (N - R)}{1 +\epsilon R} - (\gamma + \mu) \Big) \\
     & = \frac{(R - \xi_{f})(\beta+ \epsilon(\gamma + \mu) )}{1 +\epsilon R} (\xi_{f} - R) \\
     & = -\frac{(R - \xi_{f})^2 (\beta+ \epsilon(\gamma + \mu))}{1 +\epsilon R}. \\
\end{align*}
\noindent Thus, $L_{f}V(R) \le 0$ for all $R \in D - \{0\},$ and moreover, $$\mathcal{A} := \{R \in D - \{0\} : L_{f}V(R) = 0\} = \{\xi_{f}\}.$$
Then, by the fractional LaSalle theorem (Lemma 4.6 in \cite{huo2015effect}), we have that $R = \xi_{f}$ is a globally asymptotically stable equilibrium point of (\ref{model-fra}). Then,  $\displaystyle\lim_{t \to \infty} R(t; R_0,\alpha) = \xi_{f}$ for any $R_{0} \in (0, \mathcal{K}_{0}^{f})$.\\
\end{proof}


\section{Numerical experiments}
Our three proposed approaches, \eqref{model-det}–\eqref{model-fra}, can be validated numerically using data from various real-world epidemiological phenomena that align with our hypotheses. However, we drew inspiration for the phenomenon of antimicrobial resistance (AMR), specifically focusing on the resistance of {\it Escherichia coli} to colistin. Although this particular case serves as a representative study for exploring our mathematical framework, it does not fully capture the complexities inherent in AMR.
\par 
 Colistin has been the final option for treating multidrug-resistant gram-negative bacteria (MDR-GNB), and it is also extensively used in veterinary practice \cite{mendelson2018one}. The recent discovery of plasmid-mediated colistin resistance, exemplified by the mcr-1 gene in {\it E. coli}, has raised concerns regarding its use in food-producing animals and its potential to accelerate the spread of resistance. Studies have called for reassessment of colistin usage and dosing in animal husbandry to safeguard its effectiveness in human healthcare \cite{rhouma2016resistance, goyes2023management}. The reversibility of AMR, especially for colistin, has been demonstrated in laboratory experiments using {\it E. coli} carrying the mcr-1 gene \cite{wu2021reversing}. Under colistin-free conditions, there is a significant reduction in antibiotic resistance genes due to the elimination of MDR plasmids. This suggests that the high fitness costs associated with mobile genetic elements, such as plasmids, make resistance unstable, and consequently, strict antibiotic control can help reverse resistance driven by these genes \cite{wu2021reversing}. Other studies have also demonstrated that {\it E. coli} adapts to prolonged antimicrobial exposure through genetic regulation of porins and efflux pumps, which are crucial for bacterial resistance. A coordinated increase in genes encoding efflux pump transporters and a decrease in porin expression \cite{viveiros2007antibiotic}. Therefore, both transcriptional and post-translational regulation of membrane proteins are essential for physiological adaptation of gram-negative bacteria to antibiotic stress \cite{viveiros2007antibiotic}. 
\par
 We begin with a population of $N = 1e6$ bacteria and consider two initial conditions: $R(0) = N-1$ (indicating that almost all bacteria are resistant) to demonstrate extinction (or clearance of bacteria) and $R(0) = 1$ (where almost all bacteria are susceptible) to illustrate persistence. Under stressful conditions, the turnover rate of \textit{E. coli} ($\mu$) is typically significantly lower than that under optimal conditions because of the adverse factors affecting bacterial growth. In our study, we assumed a natural turnover rate of $\mu=0.1$, implying that the bacterial population doubles or clears approximately every 10 days. The population-dependent rate of R-plasmid transfer through conjugation $\beta$ for {\it E. coli} under colistin exposure was derived from observational data in \cite{wu2021reversing}, where it was suggested that a susceptible bacterium acquires an R-plasmid approximately every two days. The rate of R-plasmid loss was estimated to be within the range of $[0, 2]$, indicating that resistant bacteria may lose plasmids every 12 hours ($\gamma=2$), or never lose them ($\gamma=0$). The saturation parameter $1/\epsilon$ which represents the number of resistant bacteria at which plasmid transfer becomes less efficient due to competition between sensitive and resistant bacteria, was assumed to be equal to the total population $N$ ($\epsilon=1/N=1e-6$).  
\par 
 Finally, considering that AMR largely stems from an evolutionary process in which bacteria gradually adapt to their environment under antibiotic stress \cite{baquero2021evolutionary}, it is possible to capture AMR dynamics by highlighting the influence of historical exposure in the presence of fitness costs associated with plasmids. In this context, small values of the parameter  $\alpha$ (fractional derivative order) imply that the dynamics of plasmid-mediated resistance in bacteria exhibit significant memory effects. This suggests that the resistance characteristics of {\it E. coli} are not only dependent on current conditions, but are also heavily influenced by their previous exposure to antibiotics (selective pressure) as well as the presence of plasmids. For $\alpha \in (0,1)$, when $\alpha$ is closer to zero, the influence of past states becomes more pronounced, indicating that previous antibiotic pressures contribute significantly to the current state of resistance. 
\par 
Table \ref{table1} shows the values of the parameters used for the purpose of these experiments. 

\begin{table}[H]
\centering
\begin{tabular}{llll}
  \hline
  Parameter & Description & Dimension &Value  \\ \hline 
  $N$  &Constant population &Population& 1e6 \\
 $\mu$  &Turnover rate of bacteria &Time$^{-1}$ & 0.1 \\
 $\gamma$ &Rate of R--plasmids loss &Time$^{-1} $ &   [0, 2]  \\
 $\beta$ & Rate of  R-plasmid acquisition through conjugation &(Population $\times$ Time)$^{-1}$ &  5e-7 \\
 $\epsilon$ & Saturation rate & Population$^{-1}$ &1e-6\\
 $\sigma$ & Random perturbation for the parameter $\beta$ &Population$^{-1}$ $\times$ Time$^{-1/2}$ &[1e-7, 5e-6]\\
 $\alpha$ & Order of the fractional derivative &Dimensionless &(0,1) \\
  \hline 
\end{tabular}
\caption{Parameters involved  in the mathematical model \eqref{model-gen}: Description, dimension and values. }
\label{table1}
\end{table}
 Figure \ref{fig-extinc} illustrates the extinction scenarios for antibiotic-resistant bacteria for the three different modelling approaches (deterministic, fractional, and stochastic). These scenarios were evaluated for varying values of the plasmid loss rate $\gamma$. In each of these plots, it can be observed that the resistant bacterial populations tended to decline over time, provided that the threshold $\mathcal{K}_0$ parameters remained below 1. In the deterministic and fractional models, the extinction of resistant bacteria occurs at a faster rate as $\gamma$ increases from 1 to 2. This behaviour suggests a strong correlation between the plasmid loss rate and the speed of bacterial extinction, with higher values of $\gamma$ leading to a more rapid approach toward zero population levels for resistant strains. The stochastic model, while exhibiting a more erratic pattern due to inherent randomness, also shows a trend toward the extinction of resistant bacteria when the critical threshold $\mathcal{K}_{0}^{s}$ remains below 1. This threshold condition indicates that the resistant bacteria cannot sustain their population over time, reinforcing the consistency of the extinction behaviour across all three modelling approaches under these conditions.

\begin{figure}[H]
\centering
\subfigure[Deterministic]{
\includegraphics[width=7.5cm, height=5cm]{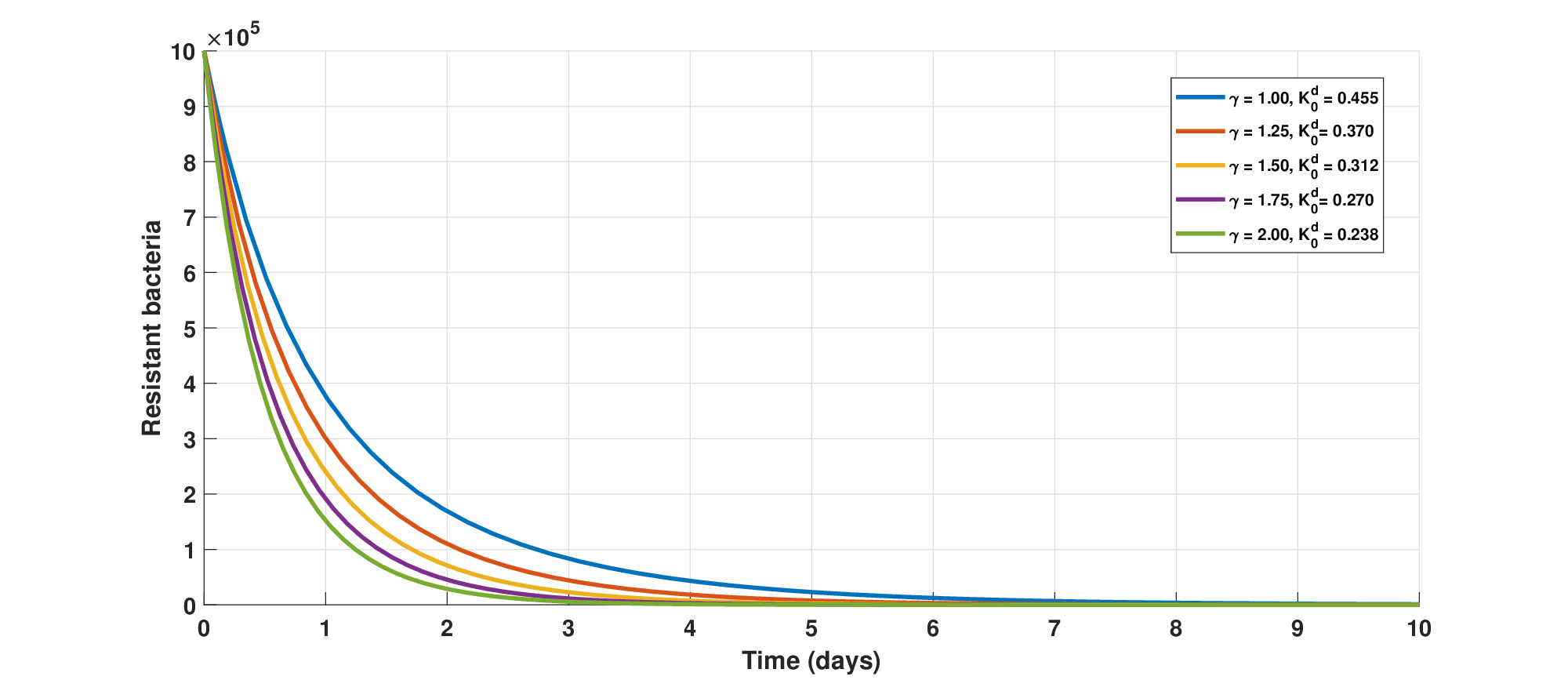}}
\subfigure[Stochastic]{
\includegraphics[width=7.5cm, height=5cm]{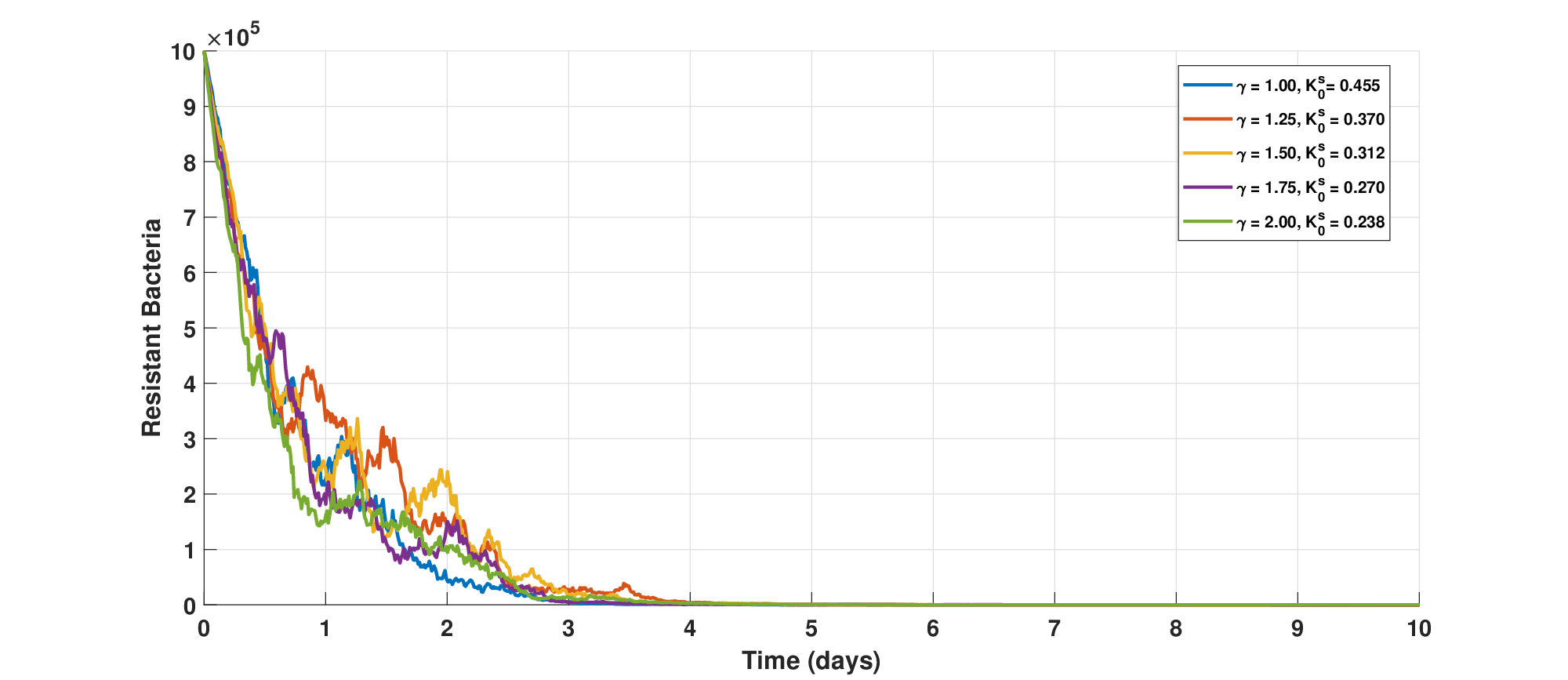}}
\subfigure[Fractional]{
\includegraphics[width=7.5cm, height=5cm]{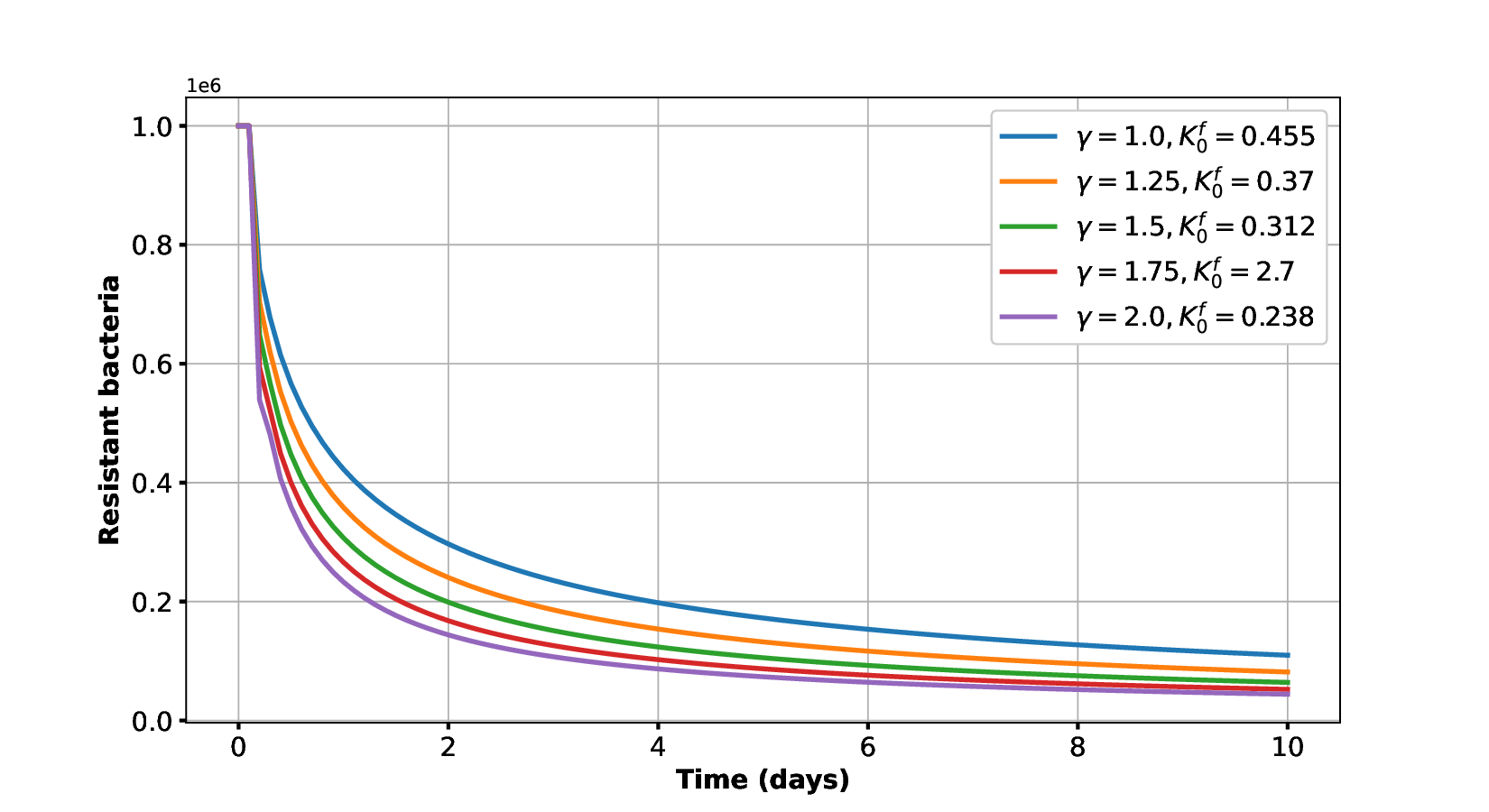}}
\caption{Extinction scenery for the three approaches and different values of $\gamma=$1, 1.25, 1.5, 1.75, 2. For the stochastic case  the perturbation parameter was $\sigma=1e-6$. For the fractional case the derivative order was $\alpha=0.7$. The initial condition $R(0)=N-1$. }
\label{fig-extinc}
\end{figure}
 Figure \ref{fig-persist} illustrates the persistence scenario for resistant bacteria using the three modelling approaches. This condition, $\mathcal{K}_0 > 1$, signifies that the resistant population can sustain itself over time, thereby avoiding a clearance. In each model, as the plasmid loss rate $\gamma$ increased within the range of 0–1, resistant bacteria exhibited a stable persistence level. This indicates that lower values of $\gamma$ promoted plasmid retention, thereby supporting the maintenance of antibiotic resistance within the bacterial population. Notably, for values of $\gamma$ near one, the population of resistant bacteria approached a steady-state level early on and remained nearly constant over extended time periods. This steady-state persistence across all three models suggests that, when $\mathcal{K}_0 > 1$ and $\gamma \leq 1$, the conditions facilitate a resilient resistant population capable of persisting over time.                                                    
\begin{figure}[H]
\centering
\subfigure[Deterministic]{
\includegraphics[width=7.5cm, height=5cm]{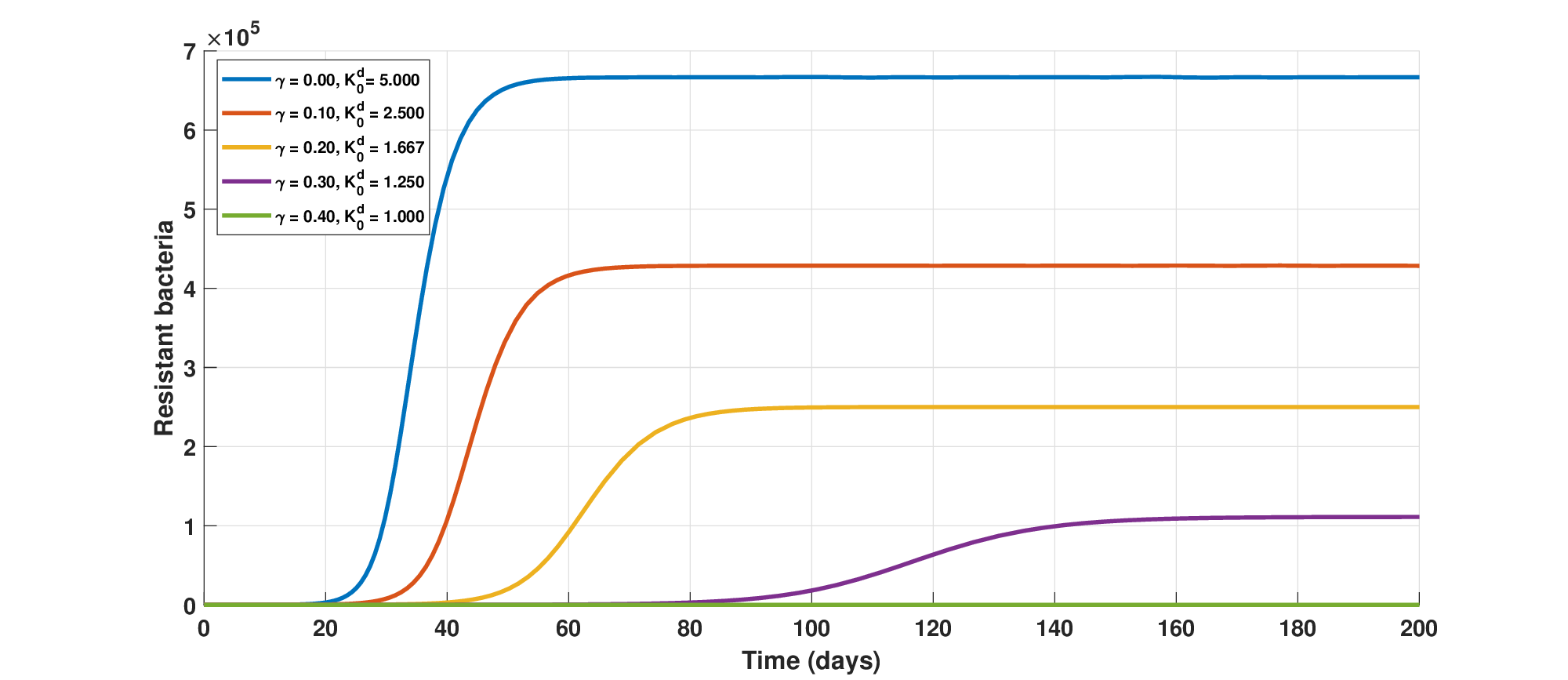}}
\subfigure[Stochastic]{
\includegraphics[width=7.5cm, height=5cm]{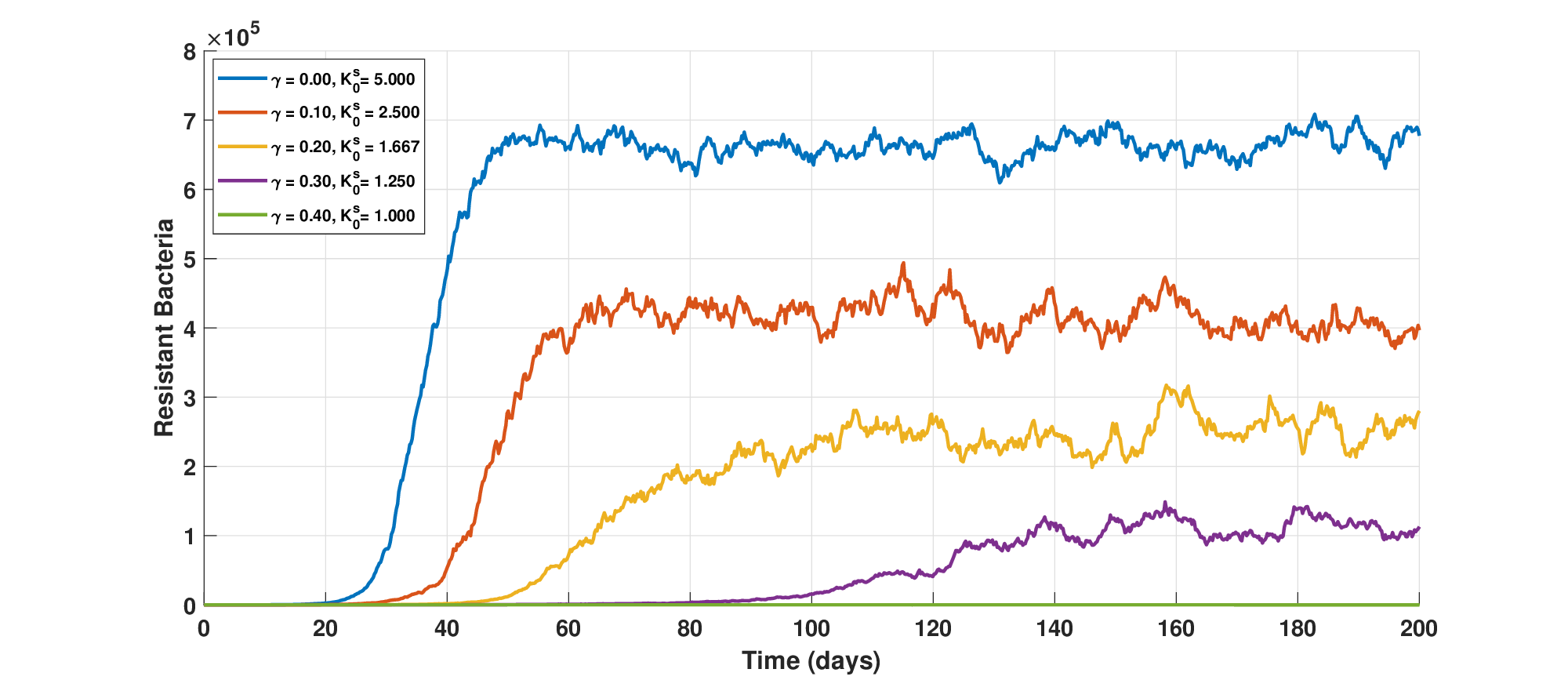}}
\subfigure[Fractional]{
\includegraphics[width=7.5cm, height=5cm]{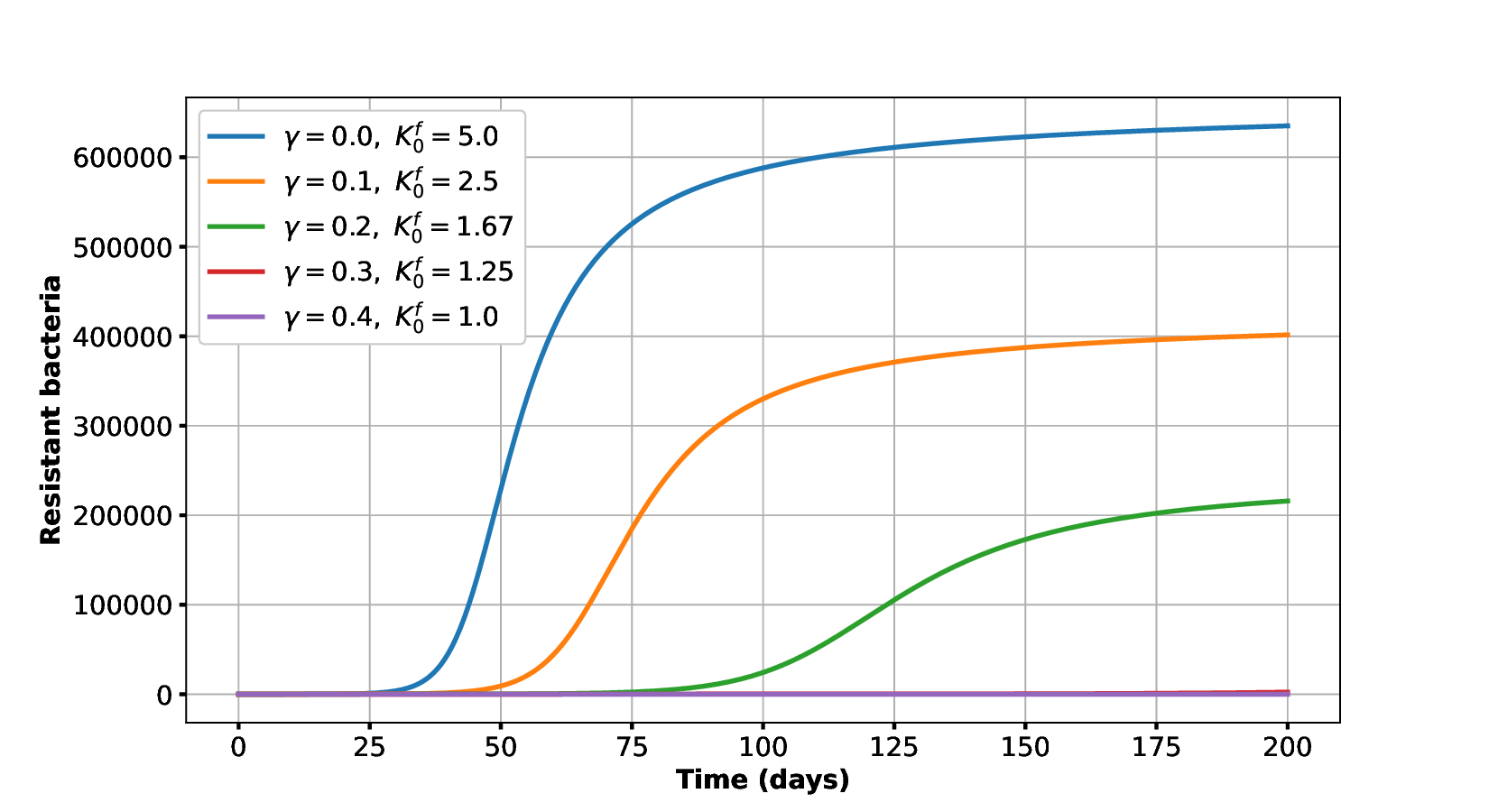}}
\caption{Persistence scenery for the three approaches and different values of $\gamma=$0, 0.1, 0.2, 0.3, 0.4. For the stochastic case  the perturbation parameter was $\sigma=1e-6$. For the fractional case the derivative order was $\alpha=0.7$. The initial condition $R(0)=1$.}
\label{fig-persist}
\end{figure}
 Figure \ref{fig-stoch-sigma} presents a comparison between the deterministic and stochastic models for different values of the white noise intensity $\sigma$. In these simulations, the plasmid loss rate was set to $\gamma = 2$ with an initial condition of $R(0) = N - 1$ for the extinction scenario, and $\gamma = 0$ with $R(0) = 1$ for the persistence scenario, where $N$ represents the total bacterial population. The plots reveal how the stochastic noise intensity $\sigma$ influences the extinction and persistence behaviours in the stochastic model compared with the deterministic model. As the value of $\sigma$ increases, the stochastic model exhibits greater fluctuations, reflecting the impact of random perturbations on bacterial population dynamics. This variability contrasts with the deterministic model, which follows a smooth trajectory toward either extinction or persistence, depending on the initial conditions and plasmid loss rate. Specifically, under the extinction condition ($\gamma = 2$, $R(0) = N - 1$), the stochastic model shows a more erratic decline in resistant bacteria as $\sigma$ increases, suggesting that a higher noise intensity accelerates fluctuations toward extinction. Conversely, in the persistence scenario ($\gamma = 0$, $R(0) = 1$), the stochastic model demonstrated sustained fluctuations around a stable population level. Here, increasing the noise intensity $\sigma$ introduces variability, but does not prevent the resistant population from persisting. These results indicate that stochastic fluctuations modulated by $\sigma$ can significantly alter the dynamics of the resistant bacteria. Higher noise levels amplify random effects, leading to more pronounced deviations from the deterministic model outcomes, particularly affecting the stability of bacterial populations under both extinction and persistence conditions. 

\begin{figure}[H]
\centering
\subfigure[Extinction]{
\includegraphics[width=7.5cm, height=5cm]{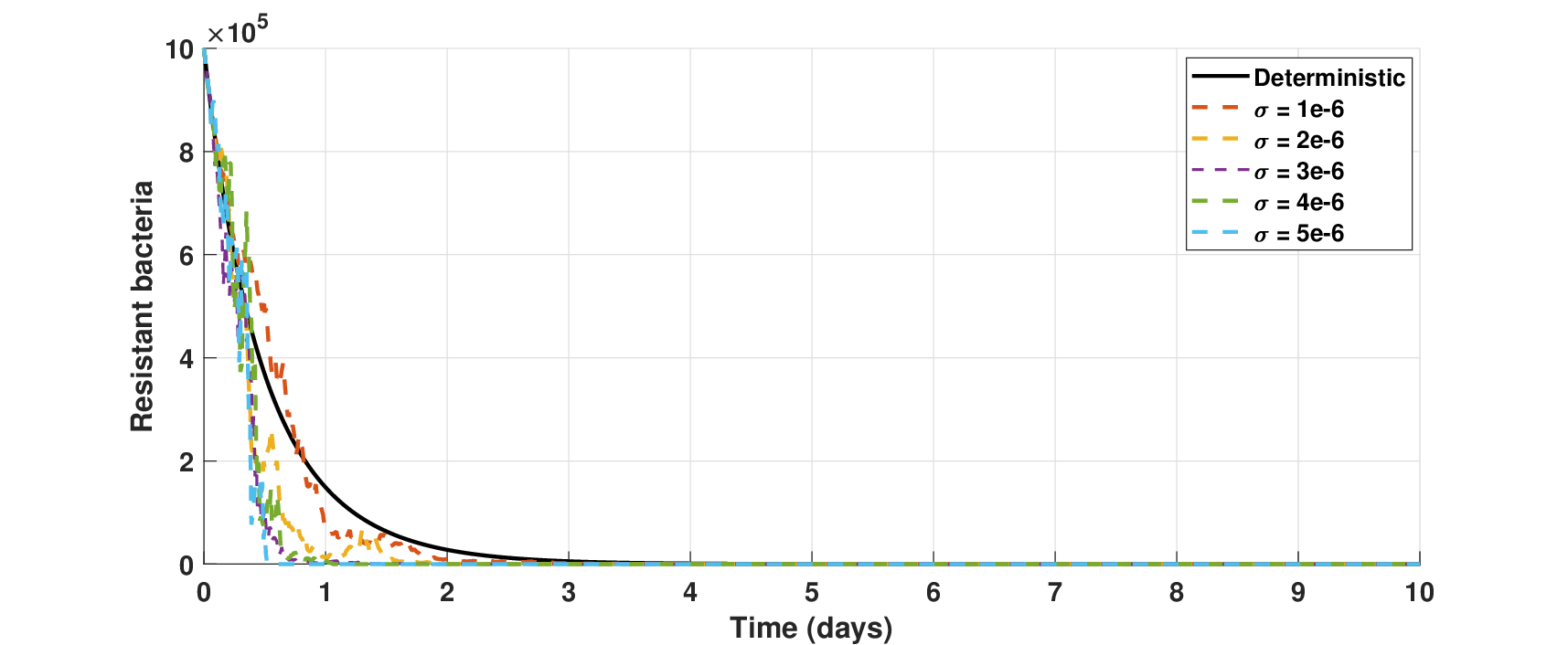}}
\subfigure[Persistence]{
\includegraphics[width=7.5cm, height=5cm]{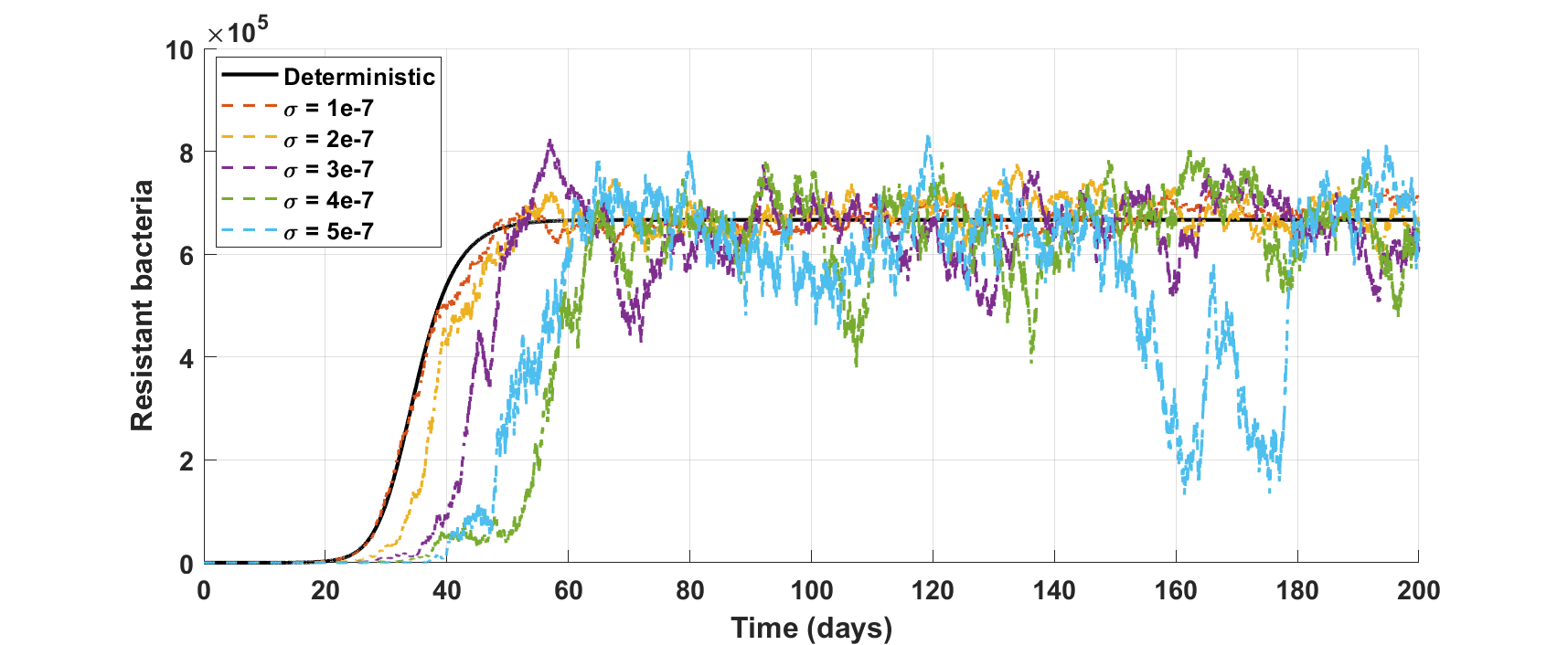}}
\caption{Deterministic and stochastic comparison for different values of $\sigma$=1e-6, 2e-6, 3e-6, 4e-6, 5e-6. Here $\gamma=2$ and $R(0)=N-1$ for extinction, and $\gamma= 0$ and $R(0)=1$ for persistence.}
\label{fig-stoch-sigma}
\end{figure} 

 In Figure \ref{fig-fractio-alpha}, the fractional model is analyzed for different values of the fractional derivative order $\alpha$. The simulations were conducted with $\gamma = 1.5$ and an initial condition of $R(0) = N - 1$ for the extinction scenario, and with $\gamma = 0.2$ and $R(0) = 1$ for the persistence scenario, where $N$ is the total bacterial population. The results show that as the order $\alpha$ of the fractional derivative increases, both the extinction and persistence behaviours in the fractional model become more pronounced. Specifically, in the extinction scenario, larger values of $\alpha$ led to a more rapid approach to zero for the resistant bacterial population. This suggests that a higher order of the fractional derivative enhances the rate at which the population declines, likely due to the increased memory effects inherent in fractional systems, which cause the population to react more dynamically to changes in the initial conditions and parameters. Conversely, in the persistence scenario, increasing $\alpha$ results in faster stabilization at the persistence level. This stabilization implies that higher values of $\alpha$ facilitate more efficient regulation of the population dynamics, allowing the resistant bacteria to reach a steady state more quickly. This behaviour is characteristic of fractional models, where memory effects become more pronounced as $\alpha$ increases, leading to faster stabilization in response to lower plasmid loss rates.

\begin{figure}[H]
\centering
\subfigure[Extinction $\gamma= 1.5$]{
\includegraphics[width=7.5cm, height=5cm]{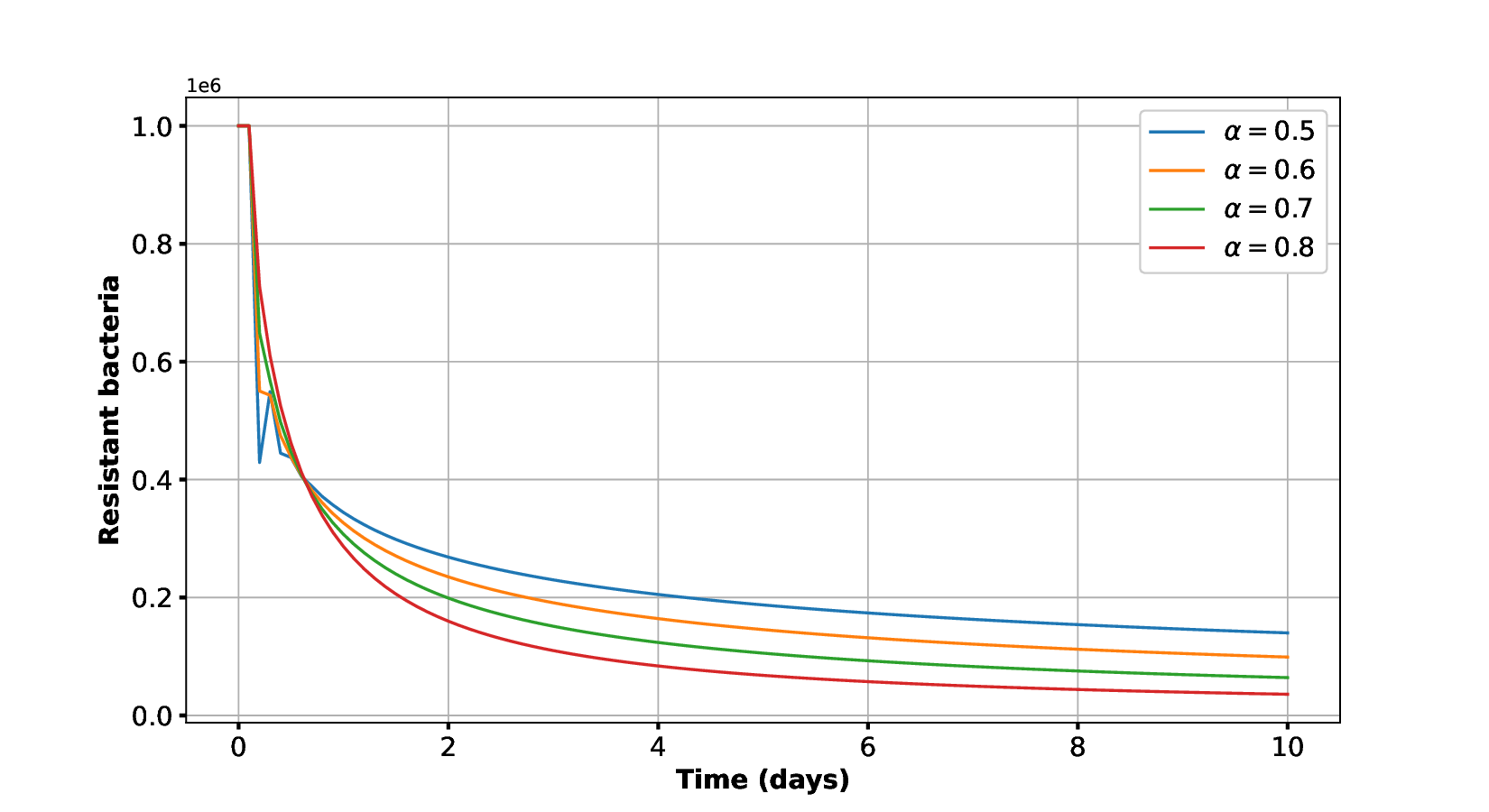}}
\subfigure[Persistence $\gamma= 0.2$ ]{
\includegraphics[width=7.5cm, height=5cm]{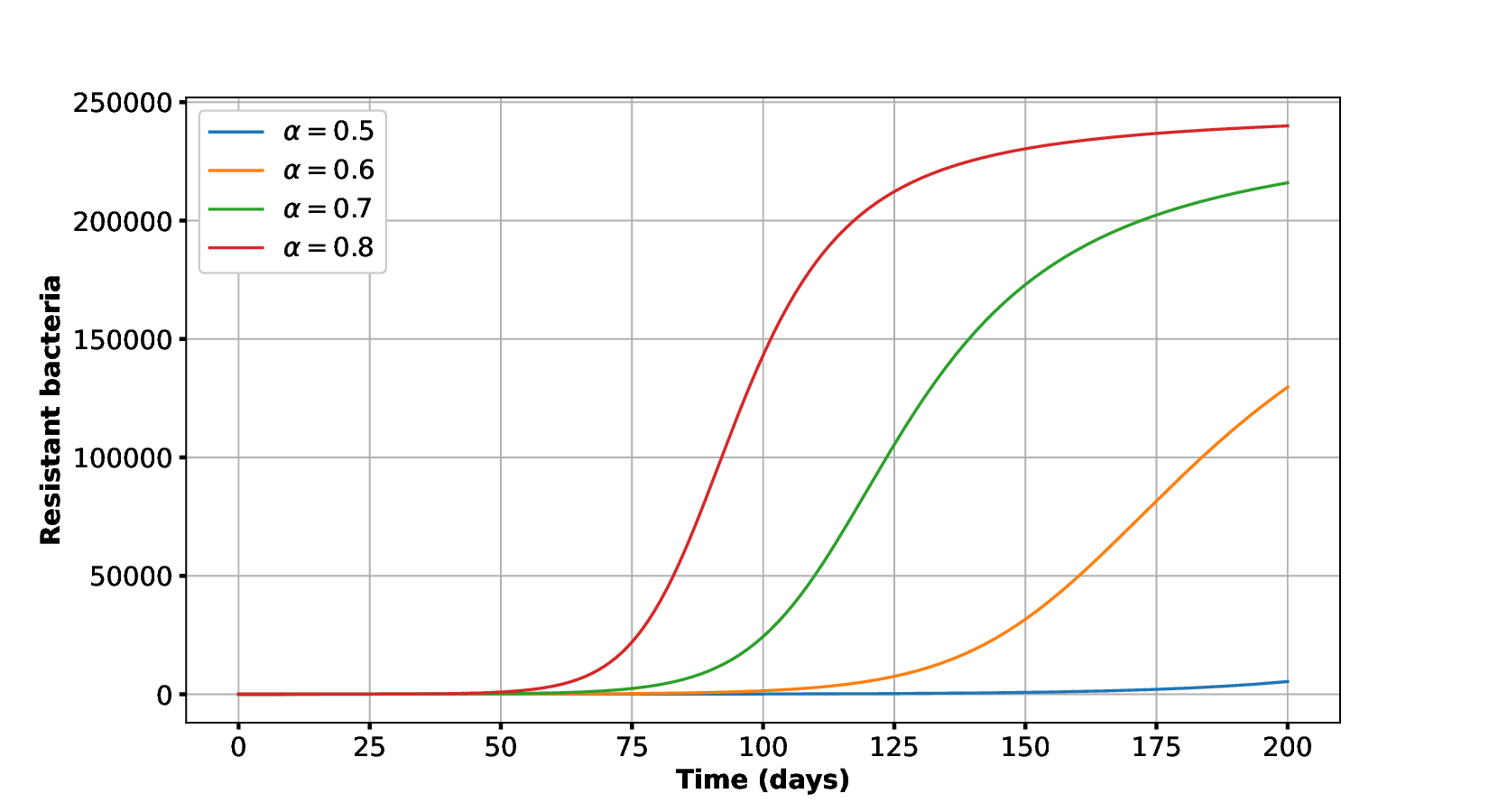}}
\caption{Fractional approach for different values of the order of the derivative $\alpha$=0.5, 0.6, 0.7, 0.8.  Here $\gamma=1.5$ and $R(0)=N-1$ for extinction, and $\gamma= 0.2$ and $R(0)=1$ for persistence. }
\label{fig-fractio-alpha}
\end{figure}
In these histograms, which show an approximation of the stationary distribution, it is illustrated that as \( \gamma \) increases, the mean of the stationary distribution of the resistant bacterial population shifts towards the origin.

\begin{figure}[H]
    \centering
    \subfigure[$\gamma=0$]{\includegraphics[width=7.5cm, height=5cm]{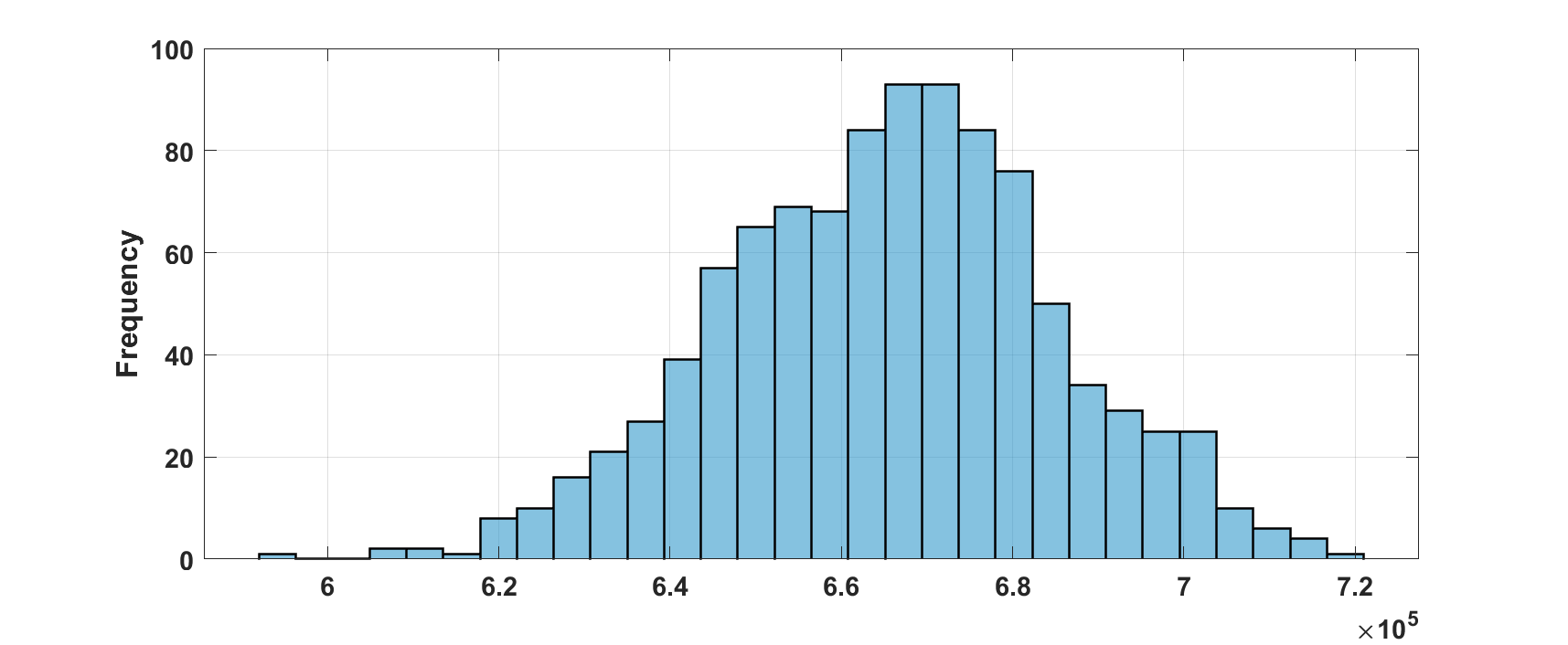}}
    \subfigure[$\gamma=0.1$]{\includegraphics[width=7.5cm, height=5cm]{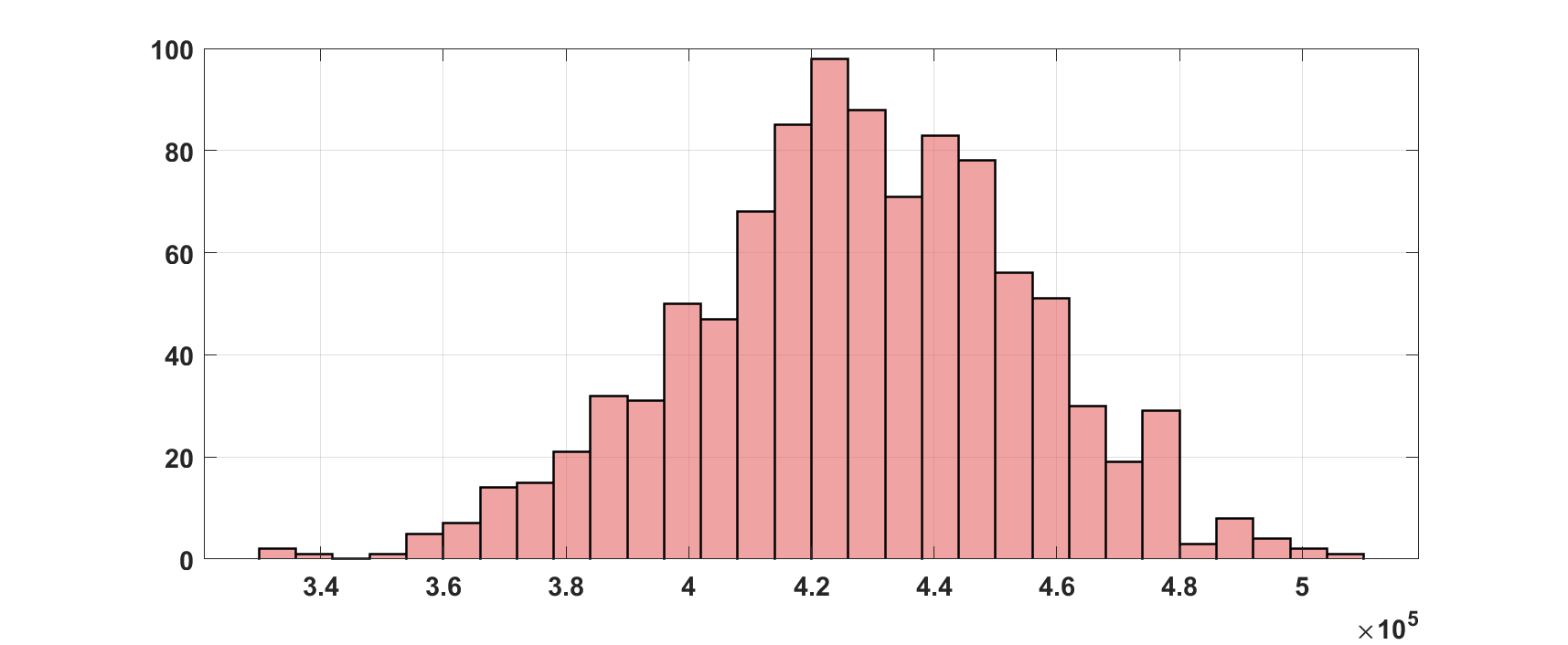}}
    \subfigure[$\gamma=0.2$]{\includegraphics[width=7.5cm, height=5cm]{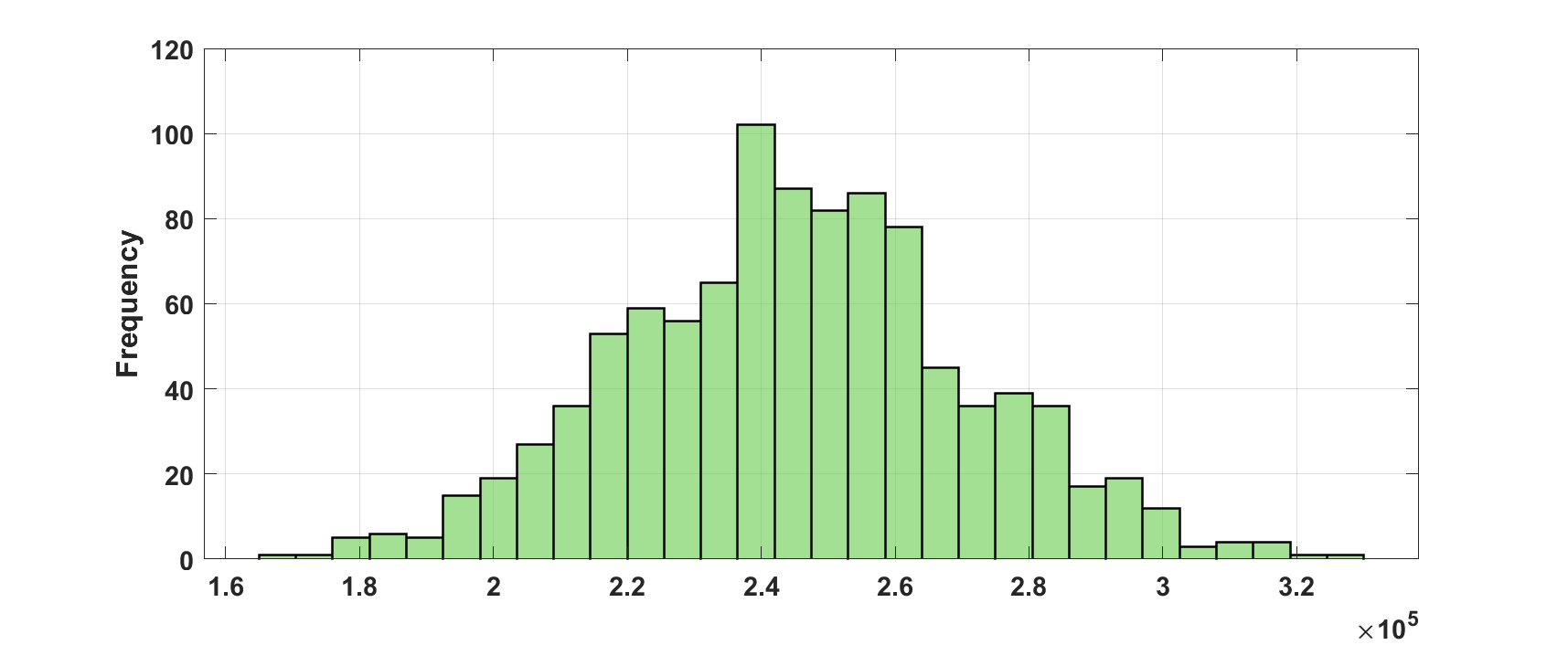}}
    \subfigure[$\gamma=0.3$]{\includegraphics[width=7.5cm, height=5cm]{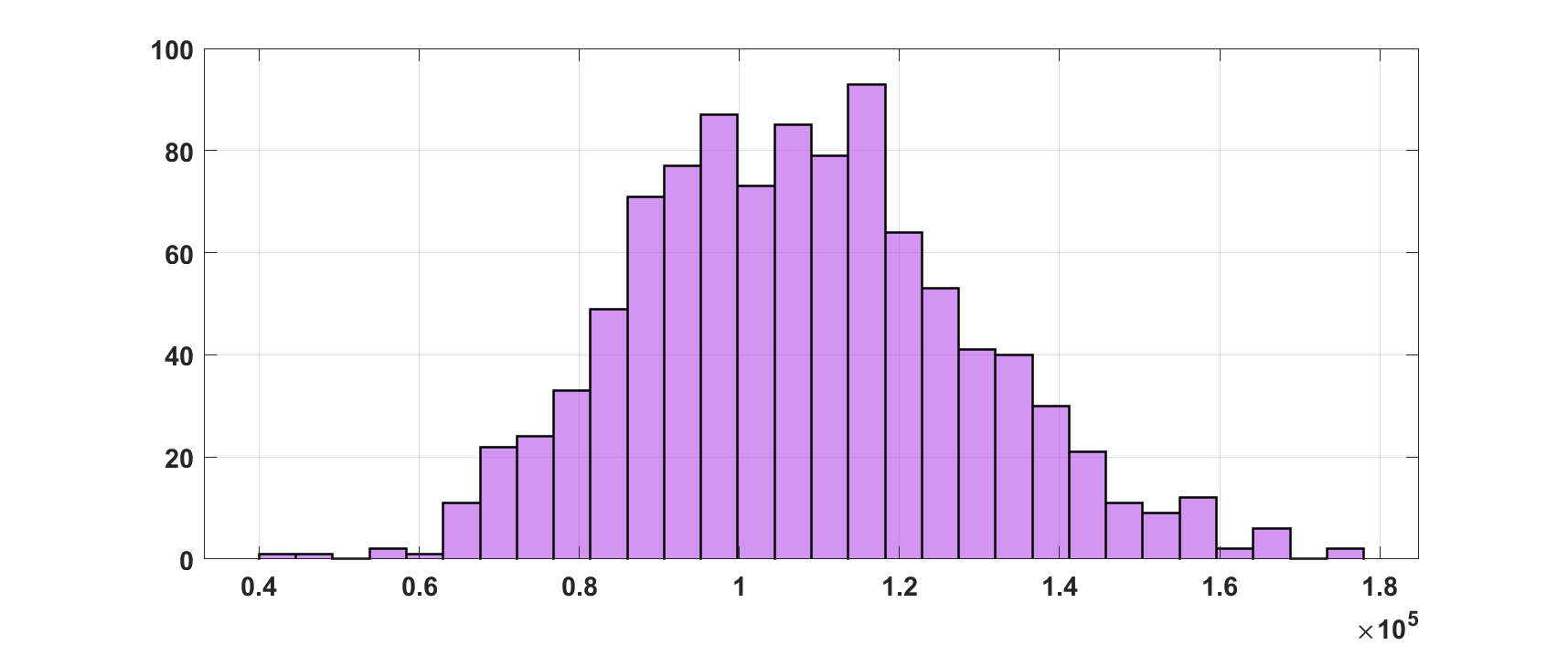}}
    \subfigure[$\gamma=0.4$]{\includegraphics[width=7.5cm, height=5cm]{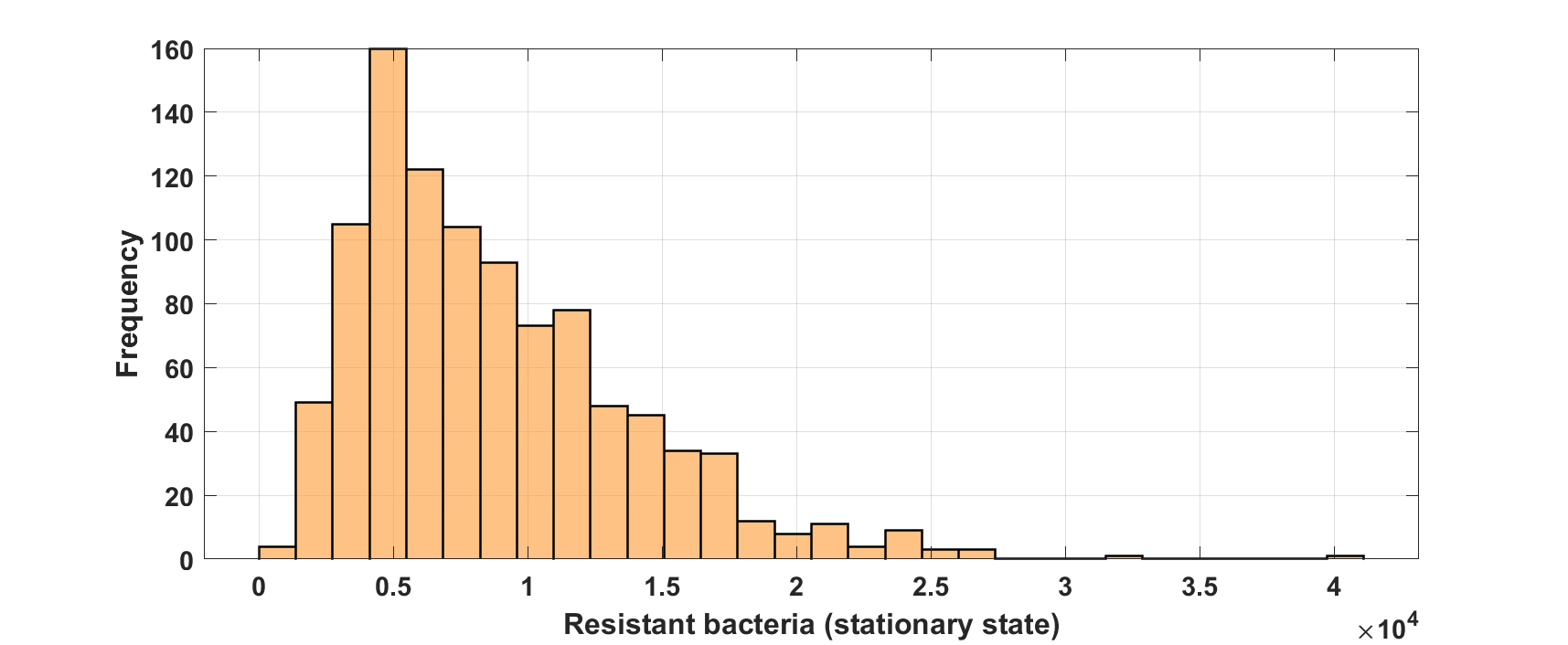}}
    \caption{Stationary mean (histograms) of the stochastic model for different values of $\gamma$ associated with the persistence of resistant bacteria. The parameter values used are those from Table \ref{table1}. Additionally, $\sigma=1e-7$ and the initial condition $R(0)=N-1$.}
    \label{fig:histogramas_gamma}
\end{figure}

\section{Discussion}
This work presented three different approaches to study a simple mathematical model that can capture the phenomenon of antimicrobial resistance (AMR) reversal, as well as other epidemiological phenomena that fit the underlying hypotheses in each of the three approaches. The main motivation for this study lies in the mathematical approach to modelling the AMR reversal. In particular, this study focused on the qualitative analysis of deterministic, stochastic, and fractional models applied to this phenomenon.
\par 
 The results obtained from the analysis in this study show that all three models—deterministic, stochastic, and fractional—are well posed from a mathematical standpoint, satisfying the conditions of existence, uniqueness, and invariance within the domain \( (0, N) \), which provides them with coherence and biological relevance. For each of these models, three thresholds \( \mathcal{K}_{0}^{d}, \mathcal{K}_{0}^{s}, \) and \( \mathcal{K}_{0}^{f} \) were identified structurally similar to the basic reproduction number characteristic of epidemiological models. It was demonstrated that, if these thresholds are less than one, the asymptotic stability of the trivial equilibrium point \( R=0 \) is guaranteed. In the stochastic model, the condition \( \mathcal{K}_{0}^{s}< 1 \) also implies that the trivial equilibrium is stable in probability. In contrast, when these thresholds are greater than one, both the deterministic and fractional models guarantee the asymptotic stability of the nontrivial equilibrium point. 
\par 
 In relation to the phenomenon of extinction or bacterial clearance, it should be noted that the modelling in the deterministic and fractional cases is similar. The stochastic model was more detailed and presented more scenarios. The existence of the stationary measure in the stochastic case provides a probabilistic framework that allows for a better understanding of the evolution of AMR. 
\par 
 To validate the theoretical models, our numerical experiments focused on the antimicrobial resistance (AMR) of \textit{Escherichia coli} to colistin, a last-resort antibiotic for infections caused by multidrug-resistant gram-negative bacteria (MDR-GNB). By analyzing the extinction and persistence of resistant bacteria in relation to different plasmid loss rates (\(\gamma\)) and the intensity of stochastic noise (\(\sigma\)), we observed that the threshold parameter \(\mathcal{K}_0^{\cdot}\) is critical for determining whether a resistant population goes extinct or persists. When \(\mathcal{K}_0^{\cdot} < 1\) and \(\gamma\) increase, plasmid loss becomes more frequent and accelerates the extinction of resistant bacteria, as the cost of maintaining the plasmid outweighs the adaptive benefits. In contrast, under persistent conditions (\(\mathcal{K}_0^{\cdot} > 1\) and \(\gamma \leq 1\)), bacteria manage to maintain the plasmid, showing early stabilization in both deterministic and fractional models as well as a persistent level of resistance in the stochastic model, even with small fluctuations. This suggests that plasmid retention is key to maintaining resistance in an environment with continuous antibiotic exposure, where a lower plasmid loss rate allows resistance to persist over time.  
\par 
 Furthermore, by incorporating the memory of prior exposure to antibiotics (evolutionary pressure) through the fractional approach, represented by the parameter \(\alpha\), it was observed that higher values of \(\alpha\) intensified the response of the bacterial population. In the extinction scenario, high values of \(\alpha\) facilitate a more rapid decline in the resistant population due to memory effects in the system, which reinforces the response to plasmid loss pressure. In the case of persistence, high values of \(\alpha\) allow for more efficient and rapid stabilization of the resistant population, demonstrating that the history of antibiotic exposure and the regulation of resistance genes (such as the reduction of porins or increased expression of efflux pumps) have an accumulating and relevant effect on bacterial adaptation. These findings highlight the importance of biological memory in AMR, as the adaptation of \textit{E. coli} to antibiotic stress is not solely dependent on current conditions, but is strongly influenced by the history of prior exposure, implying that antibiotic control strategies could leverage this effect to reverse resistance in bacterial populations. 
\par
 The results of this study provide a strong basis for the qualitative analysis of SIS-type mathematical models in general and set the stage for future research aimed at applying these findings to broader and more complex epidemiological studies. Despite the relevance of these results, this study has some  limitations. One limitation arises from the fact that the function \(g(R) = \frac{R}{1 + \epsilon R}\) is particular. Nevertheless, in \cite{builes2024stochastic}, it was proposed in a general formulation for the stochastic case. Although we wanted to focus on a real phenomenon, the choice of such a function seems appropriate. However, without significant challenges, similar results can be obtained with functions \(h\) which possess the properties determined in \cite{builes2024stochastic}. Another important limitation is the simplification of the model regarding the environmental conditions and population structure. In our approach, we assumed a homogeneous bacterial population and constant external conditions such as nutrient availability and antibiotic pressure. However, in real-world scenarios, factors such as environmental heterogeneity, competition between different bacterial strains, and fluctuations in antibiotic exposure can significantly influence resistance dynamics. Including these factors in the model could offer a more detailed perspective and allow for better representation of the complex interactions in bacterial systems exposed to antibiotics. 
\par 
 This study raises open questions for future research. A key area is the explicit calculation of the mean and variance of the stationary measure in the stochastic model (\ref{model-stoc}). The other is to identify the specific conditions under which \( R = 0 \) becomes an unstable equilibrium point in probability.

\appendix 

\section{Appendix}

\subsection{Stability}\label{A1}
\begin{definition}\textbf{Equilibrium Point of an Ordinary Differential Equation.}\label{type}\\
   Consider the ordinary differential equation:
\begin{equation}\label{ODE}
  \frac{dx}{dt} = b(x), 
\end{equation}
where $b$ is a continuous function on a non-empty interval $I \subset \mathbb{R}$, and let $\xi \in I$. We say that $\xi$ is an equilibrium point of (\ref{ODE}), or that $\xi$ is a trivial solution of (\ref{ODE}), if and only if $b(\xi) = 0$.
\end{definition}

\begin{definition}\textbf{Stability of an Equilibrium Point of an Ordinary Differential Equation.}\label{deterministic_type}\\
Let $\xi$ be an equilibrium point of (\ref{ODE}). We say that:
\begin{enumerate}
    \item $\xi$ is a stable equilibrium point of (\ref{ODE}) if and only if for every $\epsilon > 0$, there exists $\delta > 0$ such that for all $x_0 \in I$, with $|x_0 - \xi| < \delta$, we have:
    $$|x(t; x_0) - \xi| < \epsilon, \quad \text{for all } t \in \mathbb{R}_{0}^{+}.$$ 
    Otherwise, we say that $\xi$ is an unstable equilibrium point of (\ref{ODE}).
    
    \item $\xi$ is an asymptotically stable equilibrium point of (\ref{ODE}) if and only if $\xi$ is a stable equilibrium point of (\ref{ODE}) and there exists $\delta > 0$ such that for all $x_0 \in I$ with $|x_0 - \xi| < \delta$, we have:
    $$\lim_{t \to \infty} x(t; x_0) = \xi.$$
    
    \item $\xi$ is a globally asymptotically stable equilibrium point of (\ref{ODE}) if and only if $\xi$ is a stable equilibrium point of (\ref{ODE}) and for all $x_0 \in I$, we have:
    $$\lim_{t \to \infty} x(t; x_0) = \xi.$$   
\end{enumerate}
Here, $x(t; x_0)$ represents the global solution $x$ of (\ref{ODE}) that starts at $x_0$ at time $t$.
\end{definition}

\begin{definition}\textbf{Equilibrium Point of a Stochastic Differential Equation.}\label{type}\\
Consider the stochastic differential equation:
\begin{equation}\label{SDE}
  dX = b(X) \, dt + \sigma(X) \, dB(t), 
\end{equation}
\noindent where $b$ and $\sigma$ are continuous functions on an interval $I \subset \mathbb{R}$, $B$ is a Brownian motion on a probability space $(\Omega, \mathcal{F}, \mathbb{P})$, and let $\xi \in I$. We say that $\xi$ is an equilibrium point of (\ref{SDE}), or that $\xi$ is a trivial solution of (\ref{SDE}), if and only if $b(\xi) = \sigma(\xi) = 0$.
\end{definition}

\begin{definition}\textbf{Stability of an Equilibrium Point of a Stochastic Differential Equation.}\label{type_stochastic}\\
Let $\xi$ be an equilibrium point of (\ref{SDE}). We say that:
\begin{enumerate}
    \item $\xi$ is a stable equilibrium point in probability of (\ref{SDE}) if and only if for every $\epsilon > 0$ and $r \in (0, 1)$ there exists a $\delta > 0$ such that for all $x_0 \in I$, $|x_0 - \xi| < \delta$, it holds that $$\mathbb{P}\Big(|X(\cdot, t; x_0) - \xi| < \epsilon, \text{ for all } t \in \mathbb{R}_{0}^{+}\Big) \geq 1 - r.$$ In the contrary case, we say that $\xi$ is an unstable equilibrium point of (\ref{SDE}).
    \item $\xi$ is an asymptotically stable equilibrium point in probability of (\ref{SDE}) if and only if $\xi$ is a stable equilibrium point in probability of (\ref{SDE}) and for every $r \in (0, 1)$ there exists a $\delta > 0$ such that for all $x_0 \in I$, $|x_0 - \xi| < \delta$, it holds that $$\mathbb{P}\Big(\lim_{t \to \infty} X(\cdot, t; x_0) = \xi \Big) \geq 1 - r.$$   
    \item $\xi$ is a globally asymptotically stable equilibrium point in probability of (\ref{SDE}) if and only if $\xi$ is a stable equilibrium point in probability of (\ref{SDE}) and for all $x_0 \in I$, it holds that $$\mathbb{P}\Big(\lim_{t \to \infty} X(\cdot, t; x_0) = \xi \Big) = 1.$$
\end{enumerate}
Where $X(\cdot, t; x_0)$ represents the global solution $X$ of (\ref{SDE}) that starts at $x_0$ at time $t.$
\end{definition}
\begin{definition}\textbf{Equilibrium Point of a Fractional Differential Equation.}\label{fractional_type}\\
   Consider the fractional differential equation:
\begin{equation}\label{EDF}
  \frac{d^{\alpha}x}{dt^{\alpha}} = b(x), 
\end{equation}
where $b$ is a continuous function on a non-empty interval $I$ of $\mathbb{R},$ $\alpha \in (0,1),$ and let $\xi \in I.$ We say that $\xi$ is an equilibrium point of (\ref{EDF}) or that $\xi$ is a trivial solution of (\ref{EDF}) if and only if $b(\xi) = 0.$
\end{definition}

\begin{definition}\textbf{Stability of an Equilibrium Point of a Fractional Differential Equation.}\label{type}\\
Let $\xi$ be an equilibrium point of (\ref{EDF}). We say that:
\begin{enumerate}
    \item $\xi$ is a stable equilibrium point of (\ref{EDF}) if and only if for every $\epsilon > 0$, there exists a $\delta > 0$ such that for every $x_0 \in I$, $|x_0 - \xi| < \delta$, it holds that $$|x(t; x_0, \alpha) - \xi| < \epsilon, \ \text{for all}\ t \in \mathbb{R}_{0}^{+}.$$ Otherwise, we say that $\xi$ is an unstable equilibrium point of (\ref{EDF}).
    \item $\xi$ is an asymptotically stable equilibrium point of (\ref{EDF}) if and only if $\xi$ is a stable equilibrium point of (\ref{EDF}) and there exists $\delta > 0$ such that for every $x_0 \in I$, $|x_0 - \xi| < \delta$, it holds that $$\lim _{t \to \infty} x(t; x_0, \alpha) = \xi.$$   
    \item $\xi$ is a globally asymptotically stable equilibrium point of (\ref{EDF}) if and only if $\xi$ is a stable equilibrium point of (\ref{EDF}) and for every $x_0 \in I$, it holds that $$\lim _{t \to \infty} x(t; x_0, \alpha) = \xi.$$   
\end{enumerate}
where $x(t; x_0, \alpha)$ represents the global solution $x$ of (\ref{EDF}) that starts at $x_0$ at time $t.$
\end{definition}


\section*{Funding}
This research was supported by FAPESP
(Grant Number 2022/08948-2) and Universidad de Antioquia (project number 2023-58830)

\section*{Data availability}
There are no data in this study. 

\section*{Conflict of interest}
The authors have disclosed no conflicts of interest. 
\bibliographystyle{abbrv}
\bibliography{sample}


\end{document}